\documentclass{amsart}
\usepackage{amsmath,amsthm}
\usepackage{amsfonts,amssymb}
\usepackage{enumerate}
\usepackage{mathrsfs}
\usepackage{tikz}
\usepackage{float}
\copyrightinfo{2018}{H. De Bie \emph{et al.}}

\newtheorem{theorem}{Theorem}
\newtheorem{proposition}{Proposition}
\newtheorem{lemma}{Lemma}

\newtheorem{corollary}{Corollary}

\theoremstyle{definition}
\newtheorem{definition}{Definition}
\newtheorem{example}{Example}

\theoremstyle{remark}
\newtheorem{remark}{Remark}

 \newcommand{\mC}{\mathbb{C}}

 \newcommand{\cA}{\mathcal{A}}
 
 \newcommand{\cU}{\mathcal{U}}
 
 \newcommand{\Zn}{\mathbb{Z}_2^n}

\newcommand{\ospq}{\mathfrak{osp}_q(1\vert 2)}
\newcommand{\Uqsl}{\mathcal{U}_Q(\mathfrak{sl}_2)}
\renewcommand{\j}{\mathbf{j}}
\newcommand{\h}{\mathbf{h}_m}
\newcommand{\harg}[1]{\mathbf{h}_{#1}}

\begin{document}
	
\title
{The higher rank $q$-deformed Bannai-Ito and Askey-Wilson algebra}
\author{Hendrik De Bie}
\email{Hendrik.DeBie@UGent.be}
\address{Department of Mathematical Analysis\\Faculty of Engineering and Architecture\\Ghent University\\Building S8, Krijgslaan 281, 9000 Gent\\ Belgium.}

\author{Hadewijch De Clercq}
\email{Hadewijch.DeClercq@UGent.be}
\address{Department of Mathematical Analysis\\Faculty of Engineering and Architecture\\Ghent University\\Building S8, Krijgslaan 281, 9000 Gent\\ Belgium.}

\author{Wouter van de Vijver}
\email{Wouter.vandeVijver@UGent.be}
\address{Department of Mathematical Analysis\\Faculty of Engineering and Architecture\\Ghent University\\Building S8, Krijgslaan 281, 9000 Gent\\ Belgium.}

\date{\today}

\begin{abstract}
	The $q$-deformed Bannai-Ito algebra was recently constructed in the threefold tensor product of the quantum superalgebra $\mathfrak{osp}_q(1\vert 2)$. It turned out to be isomorphic to the Askey-Wilson algebra. In the present paper these results will be extended to higher rank. The rank $n-2$ $q$-Bannai-Ito algebra $\mathcal{A}_n^q$, which by the established isomorphism also yields a higher rank version of the Askey-Wilson algebra, is constructed in the $n$-fold tensor product of $\mathfrak{osp}_q(1\vert 2)$. An explicit realization in terms of $q$-shift operators and reflections is proposed, which will be called the $\mathbb{Z}_2^n$ $q$-Dirac-Dunkl model. The algebra $\mathcal{A}_n^q$ is shown to arise as the symmetry algebra of the constructed $\mathbb{Z}_2^n$ $q$-Dirac-Dunkl operator and to act irreducibly on modules of its polynomial null-solutions. An explicit basis for these modules is obtained using a $q$-deformed $\mathbf{CK}$-extension and Fischer decomposition.
\end{abstract}

\maketitle
	
\tableofcontents

\section{Introduction}
\setcounter{equation}{0}

The Bannai-Ito algebra is an associative algebra over $\mC$  with three generators $K_{12}, K_{13}, K_{23}$ and quadratic relations
\begin{align}
\label{BIrels}
\{K_{12}, K_{23}\} = K_{13} + \alpha_{13}, \quad  \{K_{12}, K_{13}\} = K_{23} + \alpha_{23}, \quad  \{K_{13}, K_{23}\} = K_{12} + \alpha_{12}.
\end{align}
Here $\{ A, B \} = AB + BA$ is the anticommutator and $\alpha_{ij}$ are structure constants. It was first introduced in \cite{Tsujimoto&Vinet&Zhedanov-2012} to encode the bispectral structure of the Bannai-Ito orthogonal polynomials, which were used in \cite{Bannai} for the classification of association schemes. Finite-dimensional irreducible representations of this algebra have been classified \cite{H} and there is a deep connection with Leonard pairs \cite{Brown}. The Bannai-Ito algebra is moreover isomorphic to a degeneration of the double affine Hecke algebra of type $(C_1^{\vee},C_1)$, see \cite{GHecke}.

For the present paper, the connection of (a central extension of) the Bannai-Ito algebra $\cA_3$ with the Lie superalgebra  $\mathfrak{osp}(1|2)$ is of crucial importance. Indeed, $\cA_3$ can be realized within the threefold tensor product of the universal enveloping algebra $\cU(\mathfrak{osp}(1|2))$ as follows. Denote by  $\Gamma$ the Casimir operator of $\mathfrak{osp}(1|2)$ and by \(\Delta\) the coproduct on this algebra. Then taking
\[
K_{12} = \Delta(\Gamma)\otimes 1, \quad K_{23} = 1\otimes \Delta(\Gamma),
\]
the relations (\ref{BIrels}) will be met using appropriate definitions for \(K_{13}\) and \(\alpha_{ij}\), see \cite{2015_Genest&Vinet&Zhedanov_CommMathPhys_336_243}.

The previous tensor product construction can be made very explicit by considering the 3-dimensional Dirac-Dunkl operator with $\mathbb{Z}_2^3$ reflection group. In \cite{DBCMP}, the Bannai-Ito algebra $\cA_3$ appears as symmetry algebra of this Dirac operator. In subsequent work \cite{DBAdv}, this led to the construction of a higher rank Bannai-Ito algebra $\cA_n$ as the symmetry algebra of the $n$-dimensional $\Zn$ Dirac-Dunkl operator. The  connection with $n$-fold tensor products of the Lie superalgebra $\cU(\mathfrak{osp}(1|2))$ is as follows: consider again the Casimir operator $\Gamma$ of $\mathfrak{osp}(1|2)$. Using the coproduct, this Casimir operator can, for a subset $A \subseteq [n]=\{ 1, 2, \ldots, n \}$, be spread out over the components of the tensor product corresponding to this subset. This yields for every subset $A$ an intermediate Casimir operator $\Gamma_A$. The resulting operators, while highly complicated, satisfy the following elegant relations
\begin{align}
\label{BI-Relations}
\{\Gamma_{A},\Gamma_{B}\}=\Gamma_{(A\cup B)\setminus (A\cap B)}+2\,\Gamma_{A\cap B}\Gamma_{A\cup B}+2\,\Gamma_{A\setminus (A\cap B)}\Gamma_{B\setminus(A\cap B)},
\end{align}
as shown in \cite[Proposition 4]{DBAdv}.

Note that similar results have been obtained for the Racah algebra, first introduced in \cite{Gr} to explain the structure of the Racah polynomials which sit atop the Askey scheme of discrete orthogonal polynomials \cite{koekoek}. This algebra is closely related to the Lie algebra $\mathfrak{su}(1,1)$, which is the even subalgebra of $\mathfrak{osp}(1|2)$. A higher rank version of the Racah algebra was obtained in \cite{I} using the generic superintegrable model on the sphere and in \cite{DBW} using $n$-fold tensor products of  $\cU(\mathfrak{su}(1,1))$, concretely realized using the Laplace-Dunkl operator (see \cite{DTAMS, DX}).

The Racah algebra corresponds to the $q = 1$ case of the Askey-Wilson or Zhedanov algebra $AW(3)$ \cite{Zhedanov-1991} underlying the Askey-Wilson polynomials (or $q$-Racah polynomials), which are the most general $q$-orthogonal polynomials in the Askey scheme. 
Also the Bannai-Ito polynomials are closely connected with the Askey scheme: they appear as a suitable $q \rightarrow -1$ limit of the Askey-Wilson polynomials.

This connection led to the consideration of the threefold tensor product of $\mathfrak{osp}_q(1|2)$, the quantum algebra which stands as the \(q\)-deformation of \(\mathfrak{osp}(1\vert 2)\) extended by its grade involution, see \cite{Genest&Vinet&Zhedanov-2016, Equitable Presentation}. In that case the algebra of intermediate Casimirs satisfies relations of the form (\ref{BIrels}) where the anticommutator has to be replaced by the $q$-anticommutator
\[
\{ A,B \}_q = q^{1/2}AB + q^{-1/2}BA,
\]
with $\alpha_{ij}$ now suitable central elements.
The ensuing algebra $\cA_3^q$ is called the $q$-deformation of the Bannai-Ito algebra. Although differently presented, it is isomorphic to the universal Askey-Wilson algebra, which is a central extension \cite{Ko1, Ko2, Terwilliger-2011} of the Askey-Wilson algebra. In this context the $q$-Bannai-Ito polynomials were defined as the Racah coefficients of $\mathfrak{osp}_q(1|2)$ and the relation with the Askey-Wilson polynomials was determined.

The connection between the \(q\)-Bannai-Ito algebra and the quantum algebra \(\mathfrak{osp}_q(1\vert 2)\) is not inherent to the tensor product setting. Indeed, it was shown in \cite{Equitable Presentation} that \(\mathcal{A}_3^q\) can be expressed in terms of the equitable generators of \(\mathfrak{osp}_q(1\vert 2)\) and hence arises as its covariance algebra. Similar results have recently been obtained for the \(q = 1\) case \cite{Baseilhac}. The multifold tensor product formalism will however arise as a quintessential tool for generalization to arbitrary rank. 

This brings us to the main challenge of the present paper, namely to construct a higher rank version of the $q$-deformed Bannai-Ito algebra, which then {\em at the same time} yields a higher rank version of the Askey-Wilson algebra. This $q$-Bannai-Ito algebra $\cA_n^q$ will be constructed within the $n$-fold tensor product of $\mathfrak{osp}_q(1|2)$ as an algebra of intermediate Casimir operators $\Gamma_A^q$, again defined for any subset $A \subseteq [n]$. At this point lies the main difficulty of our work: whereas the construction of $\Gamma_A^q$ for $A$ a set of consecutive integers is relatively straightforward using the Hopf coproduct of $\mathfrak{osp}_q(1|2)$, this is no longer the case when the set $A$ shows holes. In that case an intricate sequence of extension morphisms, defined later in formulas (\ref{def:tau isomorphism}), (\ref{def: extension procedure part 1}), (\ref{def: extension procedure part 2}), has to be applied to the initial Casimir operator.

In a next step we obtain in Theorem \ref{prop: q-anticommutation Gamma} the relations for $\Gamma_A^q$ and  $\Gamma_B^q$ under some technical requirements on the sets \(A\) and \(B\), as
\[
\{\Gamma_A^q,\Gamma_B^q\}_q = \Gamma_{(A\cup B)\setminus(A\cap B)}^q + (q^{1/2}+q^{-1/2})\left(\Gamma_{A\cap B}^q\Gamma_{A\cup B}^q + \Gamma_{A\setminus (A\cap B)}^q\Gamma_{B\setminus (A\cap B)}^q\right).
\]
In the limit $q \rightarrow 1$ this relation clearly reduces to (\ref{BI-Relations}).

We prefer to work with the $q$-Bannai-Ito algebra instead of directly with the universal Askey-Wilson algebra, as the relations of the former exhibit more symmetry and are easier to manipulate.
To complete our construction, we therefore explain in Section \ref{Paragraph: Connection with the AW-algebra} how the algebra $\cA_3^q$ is isomorphic to the universal Askey-Wilson algebra. This is achieved by using the relation between $\mathfrak{osp}_q(1|2)$ and $\cU_q(\mathfrak{sl}_2)$ on the one hand, and the embedding of the universal Askey-Wilson algebra in the threefold tensor product of $\cU_q(\mathfrak{sl}_2)$ on the other hand, see \cite{Huang}. Therefore, $\cA_n^q$ can equally be considered as the higher rank Askey-Wilson algebra.

Our constructions differ from the generalized Askey-Wilson algebras obtained in \cite{Baseilhac&Koizumi-2005}, based on tensor product representations of another quantum group, namely the quantum affine algebra \(U_q(\widehat{\mathfrak{sl}_2})\). In this approach, the Askey-Wilson algebra is viewed as a quotient of the \(q\)-Onsager algebra \cite{Baseilhac2}, whereas our methods rather refer directly to the universal Askey-Wilson algebra as presented in \cite{Terwilliger-2011}.

The Bannai-Ito and Racah algebras are intimately connected with superintegrable systems, see e.g. \cite{DBW, DBsup, I}.
To showcase the power of our new algebraic approach, we therefore construct a superintegrable model related to $\cA_n^q$ which we call the $\Zn$ $q$-Dirac-Dunkl model. It is governed by the  $\Zn$ $q$-Dirac-Dunkl operator, of which we determine an algebra of symmetries given precisely by $\cA_n^q$, see Proposition \ref{prop: Joint symmetries}. We explicitly determine modules of polynomial null-solutions of this operator. We subsequently construct a basis for these modules using the familiar Fischer decomposition and Cauchy-Kowalewska extension procedure, which we derive in this context, and show that this basis diagonalizes an abelian subalgebra of $\cA_n^q$. 
Finally we show in Theorem \ref{thm: Action irreducible} that these modules form irreducible representations of $\cA_n^q$. The main technical difficulty is to determine explicitly the action of suitable generators of $\cA_n^q$ on basis vectors, as given in Theorem \ref{prop: Three-term recurrence relation}. 

In the limit $q \rightarrow 1$ the  $\Zn$ $q$-Dirac-Dunkl operator reduces to the operator defined in Section 5 of \cite{DBAdv}. Our results yield an alternative proof of the irreducibility of the modules in this limit. They should also be compared with the irreducibility result for Racah algebra modules recently obtained in \cite{I2}.

The Askey-Wilson algebra also arises frequently in the context of superintegrable systems. In \cite{Baseilhac2} a set of mutually commuting elements was constructed from iterated coproducts of Askey-Wilson generators. These elements can be used to construct non-local integrals of motion for several quantum integrable models, such as the XXZ spin chain and the sine-Gordon model. Here we extend this iteration of coproducts to an algorithm to obtain more general higher rank algebra generators.

Note that a superintegrable Gaudin system with \(\mathfrak{osp}_q(1\vert 2)\)-symmetry has also been considered in \cite{Musso} in a purely algebraic context and with slightly different conventions on the generators.

Let us finally discuss the connection of our work with the multivariate Askey-Wilson or \(q\)-Racah polynomials defined in \cite{Gasper2, Gasper} as generalizations of the work of \cite{Tratnik}. On the one hand, these polynomials appear as recoupling or $3nj$ coefficients for $n$-fold tensor products of $\mathfrak{su}_q(1,1)$, see \cite{Gqq}. In our Dirac model, this would translate to computing the connection coefficients between different bases of null-solutions of our $\Zn$ $q$-Dirac-Dunkl operator and could serve as a way to define multivariate $q$-Bannai-Ito polynomials (which then are multivariate Askey-Wilson polynomials in disguise). These bases are constructed by permuting the order in which the \(\mathbf{CK}\)-extensions act in the basis, see formula (\ref{eq: Basis for M_k}). 
On the other hand, in \cite{I1} Iliev constructs a commuting family of $q$-difference operators which diagonalize the multivariate Askey-Wilson polynomials. As we have constructed a similar abelian subalgebra of $\cA_n^q$ that diagonalizes our basis, it seems plausible that the action of the diagonal operators in \cite{I1} can be extended to an action of the full algebra $\cA_n^q$. This would moreover complement the realization of the \(q\)-Onsager algebra by Ilievs difference operators proposed in \cite{Baseilhac&Martin-2018}. These highly technical issues will be discussed in our subsequent work \cite{DeBie&DeClercq-2019}.

The paper is organized as follows. In Section \ref{Section: Tensor product approach} we construct the higher rank $q$-Bannai-Ito algebra $\cA_n^q$ in the $n$-fold tensor product of  $\mathfrak{osp}_q(1|2)$. We prove in Theorem \ref{prop: q-anticommutation Gamma} the crucial relation that exists in this algebra and use it to find a generating set in Corollary \ref{cor: Generating set}.
We also explain in detail the connection between $\cA_n^q$ and the universal Askey-Wilson algebra.
In Section \ref{Section: Z2n q-Dirac-Dunkl model} we construct a concrete realization of $\cA_n^q$ using the $\Zn$ $q$-Dirac-Dunkl model. We introduce modules of null-solutions of this \(q\)-Dirac-Dunkl operator and construct an explicit basis. In Section \ref{Section: Action of the symmetry algebra} we show that these modules are irreducible under the action of $\cA_n^q$. We end with some conclusions.
	
\section{The tensor product approach}
\label{Section: Tensor product approach}
Let \(q\) be a non-zero complex number with \(\vert q\vert \neq 1\). For \(n\in \mathbb{N}\), we denote by \([n]_q\) the \(q\)-number
\[
[n]_q = \frac{q^n-q^{-n}}{q-q^{-1}}.
\]
The same notation will be used for operators:
\[
[A]_q=\frac{q^{A}-q^{-A}}{q-q^{-1}}.
\]
We will write \([n]\) for the set \(\{1,2,\dots,n\}\) and \([i;j]\) for the set \(\{i,i+1,\dots,j\}\).
The \(q\)-anticommutator of two operators \(A\) and \(B\) is defined as
\[
\{A,B\}_q = q^{1/2}AB+q^{-1/2}BA.
\]

The quantum superalgebra \(\ospq\) is the \(\mathbb{Z}_2\)-graded unital associative algebra with generators \(A_0,A_-,A_+\) and the grade involution \(P\), satisfying the commutation relations \cite{Kulish}
\begin{align}
\label{def: ospq without K}
	\begin{split}
		[A_0,A_{\pm}] = \pm A_{\pm}, \quad \{A_+,A_-\} = [2A_0]_{q^{1/2}}, \\ [P,A_0]= 0, \quad \{P,A_{\pm}\} = 0, \quad P^2 = 1.
	\end{split}
\end{align}
The algebra \(\ospq\), sometimes also denoted \(\mathcal{U}_q(\mathfrak{osp}(1\vert 2))\), reduces to the universal enveloping algebra \(\mathcal{U}(\mathfrak{osp}(1\vert 2))\) in the limit for the parameter \(q\rightarrow 1\).
Defining the operators
\[
K = q^{A_0/2}, \quad K^{-1} = q^{-A_0/2},
\]
these relations take the equivalent form
\begin{align}
\label{def: ospq with K}
	\begin{split}
		KA_{+}K^{-1} = q^{1/2} A_{+}, \quad KA_{-}K^{-1} = q^{-1/2} A_{-}, \quad \{A_+,A_-\} = \frac{K^2-K^{-2}}{q^{1/2}-q^{-1/2}}, \\ \quad \{P,A_{\pm}\} = 0, \quad 
		\ [P,K] = 0, \quad [P,K^{-1}] = 0, \quad KK^{-1} = K^{-1}K = 1, \quad P^2 = 1.
	\end{split}
\end{align}
With these generators one can construct the following Casimir operator
\begin{equation}
	\label{def: Gamma^q}
	\Gamma^q = \left(-A_+A_-+\frac{q^{-1/2}K^2-q^{1/2}K^{-2}}{q-q^{-1}} \right)P.
\end{equation}
It is easily checked that \(\Gamma^q\) commutes with all elements of \(\ospq\). The expression between brackets in (\ref{def: Gamma^q}) is in fact the sCasimir operator of \(\ospq\), see \cite{Lesniewski}, which commutes with \(A_0\) and anticommutes with \(A_{\pm}\).

The algebra \(\ospq\) can be endowed with a coproduct \(\Delta: \ospq \rightarrow \ospq\otimes\ospq \) acting on the generators as \cite{Genest&Vinet&Zhedanov-2016}
\begin{equation}
	\label{def: Coproduct}
	\Delta(A_{\pm}) = A_{\pm}\otimes KP + K^{-1}\otimes A_{\pm}, \quad \Delta(K) = K\otimes K, \quad \Delta(P) = P\otimes P,
\end{equation}
which satisfies the coassociativity property
\begin{equation}
	\label{Coassociativity}
	(1\otimes\Delta)\Delta = (\Delta\otimes 1)\Delta.
\end{equation}
By direct computation one can express \(\Delta(\Gamma^q) \) as
\begin{align}
\begin{split}
\label{eq:Delta(Gamma)}
\Delta(\Gamma^q) = & \ -q^{1/2}\left(A_-K^{-1}P\otimes A_+K\right) + q^{-1/2}\left(A_+K^{-1}P\otimes A_-K\right) \\ & + \left[\frac12\right]_q\left(K^{-2}P\otimes K^2P\right) + \left(\Gamma^q\otimes K^2P\right) + \left(K^{-2}P\otimes \Gamma^q\right).
\end{split}
\end{align}

This coproduct, together with the counit \(\epsilon: \ospq \rightarrow \mathbb{C}\)
\begin{equation}
\label{def: counit}
\epsilon (A_{\pm}) = 0, \quad \epsilon(K)=1, \quad \epsilon(P) = 1,
\end{equation}
and the antipode \(S: \ospq \rightarrow\ospq\)
\begin{equation}
\label{def: antipode}
S(A_{\pm}) = -q^{\pm 1/2}A_{\pm}P, \quad S(K) = K^{-1}, \quad S(P) = P,
\end{equation}
gives \(\ospq\) the structure of a Hopf algebra. 
Following \cite{Genest&Vinet&Zhedanov-2016}, we will always consider the tensor product algebra \(\ospq\otimes\ospq\) with its standard product law
\begin{equation}
\label{eq: Standard product law}
(a_1\otimes a_2)(b_1\otimes b_2) = a_1b_1\otimes a_2b_2.
\end{equation}
This is in contrast to the graded product rule \((a_1\otimes a_2)(b_1\otimes b_2) = (-1)^{p(a_2)+p(b_1)}a_1b_1\otimes a_2b_2\), with \(p(x)\) the parity of \(x\), used in \cite{Kulish}. This extra use of the parity would be redundant here, as we have chosen to treat the grade involution \(P\) as a separate generator.

The rank 1 \(q\)-deformed Bannai-Ito algebra was introduced in \cite{Genest&Vinet&Zhedanov-2016} within the threefold tensor product of \(\ospq\). Three types of Casimir operators were identified, namely the initial Casimir operators
\begin{equation}
\label{def: Initial Casimirs}
	\Gamma_{\{1\}}^{q} = \Gamma^q\otimes1\otimes1, \quad \Gamma_{\{2\}}^{q} = 1\otimes \Gamma^q\otimes 1, \quad \Gamma_{\{3\}}^{q}=1\otimes1\otimes\Gamma^q,
\end{equation} 
the intermediate Casimir operators
\begin{equation}
\label{def: Intermediate Casimirs}
	\Gamma_{\{1,2\}}^q = \Delta(\Gamma^q)\otimes1, \quad \Gamma_{\{2,3\}}^q = 1\otimes\Delta(\Gamma^q),
\end{equation}
and the total Casimir operator
\begin{equation}
\label{def: Total Casimir}
	\Gamma_{\{1,2,3\}} = (1\otimes\Delta)\Delta(\Gamma^q).
\end{equation}
Defining \(\Gamma_{\{1,3\}}^q\) through the relation
\begin{equation}
\label{eq: First BI-relation}
\{\Gamma_{\{1,2\}}^q,\Gamma_{\{2,3\}}^q\}_q = \Gamma_{\{1,3\}}^q + (q^{1/2}+q^{-1/2})\left(\Gamma_{\{1\}}^q\Gamma_{\{3\}}^q + \Gamma_{\{2\}}^q\Gamma_{\{1,2,3\}}^q\right),
\end{equation}
one can show that these operators satisfy the relations
\begin{equation}
	\label{eq: q-Bannai-Ito rank 1}
	\{\Gamma_{\{i,j\}}^q,\Gamma_{\{j,k\}}^q\}_q = \Gamma_{\{i,k\}}^q + (q^{1/2}+q^{-1/2})\left(\Gamma_{\{i\}}^q\Gamma_{\{k\}}^q+\Gamma_{\{j\}}^q\Gamma_{\{i,j,k\}}^q\right),
\end{equation}
where \((ijk)\) is an even permutation of \(\{1,2,3\}\).

These relations coincide with the defining relations \cite{Tsujimoto&Vinet&Zhedanov-2012} of the Bannai-Ito algebra in the limit \(q\rightarrow 1\), hence the algebra generated by \(\Gamma_{\{1,2\}}^q\), \( \Gamma_{\{2,3\}}^q\) and \(\Gamma_{\{1,3\}}^q\) was identified as a \(q\)-deformed Bannai-Ito algebra, denoted here by \(\mathcal{A}_3^q\).
In the next paragraphs, we will introduce the correct definitions to extend this algebra to the multifold tensor product setting.

\subsection{Coaction and fourfold tensor products}

Before moving up to multifold tensor products, we will first need to enhance our understanding of the threefold tensor product case. Let us define \(\mathcal{I}\) as the subalgebra of \(\ospq\) generated by \(A_-K\), \(A_+K\), \(K^2P\) and \(\Gamma^q\). This choice of generators is motivated by the observation that \(\Delta(\Gamma^q)\in\ospq\otimes\mathcal{I}\), as follows from (\ref{eq:Delta(Gamma)}). This subalgebra has the following interesting properties.

\begin{proposition}
	\label{lemma: prop tau}
	The algebra \(\mathcal{I}\) is a left coideal subalgebra of \(\ospq\), as well as a left \(\ospq\)-comodule with coaction \(\tau: \mathcal{I}\to\ospq\otimes\mathcal{I}\) defined by
	\begin{align}
	\begin{split}
	\label{def:tau isomorphism}
	\tau(A_-K) = & \ K^2P\otimes A_-K, \\
	\tau(A_+K) = & \ (K^{-2}P\otimes A_+K)+q^{-1/2}(q-q^{-1})(A_+^2P\otimes A_-K) \\&+ q^{-1/2}(q^{1/2}-q^{-1/2})(A_+K^{-1}P\otimes K^2P)\\& +q^{-1/2}(q-q^{-1})(A_+K^{-1}P\otimes \Gamma^q), \\
	\tau(K^2P) = & \ 1\otimes K^2P - (q-q^{-1})(A_+K\otimes A_-K),\\
	\tau(\Gamma^q) = & \ 1\otimes \Gamma^q.
	\end{split}
	\end{align}
\end{proposition}
\begin{proof}
	We have to verify the following requirements:
	\begin{enumerate}[(i)]
		\item \(\mathcal{I}\) is a left coideal subalgebra, i.e.\ \(\Delta(\mathcal{I}) \subset \ospq\otimes\mathcal{I}\).
		\item The mapping \(\tau\) is a well-defined algebra morphism, i.e.\ it preserves the algebra relations in \(\mathcal{I}\).
		\item The algebra morphism \(\tau\) is a left coaction map, i.e.\ it has the properties
		\begin{align}
		\label{comodule prop 1}
		(1\otimes\tau)\tau = (\Delta\otimes 1)\tau, \\
		\label{comodule prop 2}
		(\epsilon\otimes 1)\tau \cong \mathrm{id},
		\end{align}
		where the isomorphism in the last line is the canonical identification of \(\mathbb{C}\otimes\mathcal{I}\) with \(\mathcal{I}\).
	\end{enumerate}
	The properties (i) and (iii) can easily be checked on each of the generators of \(\mathcal{I}\) using (\ref{def: Coproduct}), (\ref{def: counit}) and (\ref{def:tau isomorphism}). By (\ref{def: ospq with K}), the algebra \(\mathcal{I}\) is defined through the relations
	\begin{align*}
	(K^2P)(A_-K) &= -q^{-1}(A_-K)(K^2P), \ \ &(K^2P)(A_+K) &= -q(A_+K)(K^2P), \quad\\ \{A_+K, A_-K\}_q &= \frac{(K^2P)^2-1}{q^{1/2}-q^{-1/2}}, &[\Gamma^q, A_+K] &= [\Gamma^q, A_-K]= [\Gamma^q, K^2P] = 0,
	\end{align*}
	which are invariant under \(\tau\), as follows from a lengthy but straightforward calculation.
\end{proof}

The expressions (\ref{def:tau isomorphism}) originate from the fact that the element \(\Gamma_{\{1,3\}}^q\), defined through the relation (\ref{eq: First BI-relation}), can be written as
\[
\Gamma_{\{1,3\}}^q = (1\otimes\tau)\Delta(\Gamma^q),
\]
as one can check by a direct computation.

The mapping \(\tau\) and the coproduct \(\Delta\) will now enable us to step up to the case \(n = 4\) and define the operators \(\Gamma_A^q \in \ospq^{\otimes 4}\) for all \(A\subseteq \{1,2,3,4\}\). For sets \(A\) of consecutive integers we may apply the familiar extension procedure to obtain the initial Casimir operators
\begin{align*}
	\begin{split}
		\Gamma_{\{1\}}^q = & \,\Gamma^q\otimes1\otimes1\otimes1, \quad \Gamma_{\{2\}}^q =  1\otimes \Gamma^q\otimes1\otimes1,\\
		\Gamma_{\{3\}}^q = & \,1\otimes1\otimes\Gamma^q\otimes1, \quad \Gamma_{\{4\}}^q = 1\otimes1\otimes1\otimes\Gamma^q, 
	\end{split}
\end{align*}
the intermediate Casimir operators
\begin{equation}
	\begin{gathered}
		\Gamma_{\{1,2\}}^q = \Delta(\Gamma^q)\otimes1\otimes1, \quad \Gamma_{\{2,3\}}^q = 1\otimes\Delta(\Gamma^q)\otimes1, \quad \Gamma_{\{3,4\}}^q = 1\otimes1\otimes\Delta(\Gamma^q), \nonumber\\
		\Gamma_{\{1,2,3\}}^q = (1\otimes\Delta)\Delta(\Gamma^q)\otimes1, \quad \Gamma_{\{2,3,4\}}^q = 1\otimes(1\otimes\Delta)\Delta(\Gamma^q), 
	\end{gathered}
\end{equation}
and the total Casimir operator
\begin{equation*}
	\Gamma_{\{1,2,3,4\}}^q = (1\otimes1\otimes\Delta)(1\otimes\Delta)\Delta(\Gamma^q).
\end{equation*}
For \(A = \emptyset\) we will use the scalar element
\[
\Gamma_{\emptyset}^q = -\left[\frac12\right]_q.
\]
We can also construct new intermediate Casimir operators, corresponding to sets of non-consecutive integers, i.e. sets with \emph{holes}. This is done using the coaction \(\tau\), which creates these holes:
\begin{equation}
	\begin{gathered}
		\label{def:q-Bannai-Ito rank 2}
		\Gamma_{\{1,3\}}^q = \left((1\otimes\tau)\Delta(\Gamma^q)\right)\otimes 1, \quad \Gamma_{\{2,4\}}^q = 1\otimes\left((1\otimes\tau)\Delta(\Gamma^q)\right), \\
		\Gamma_{\{1,4\}}^q = (1\otimes\Delta\otimes1)(1\otimes\tau)\Delta(\Gamma^q), \\
		\Gamma_{\{1,2,4\}}^q = (1\otimes1\otimes\tau)(1\otimes\Delta)\Delta(\Gamma^q), \quad \Gamma_{\{1,3,4\}}^q = (1\otimes1\otimes\Delta)(1\otimes\tau)\Delta(\Gamma^q).
	\end{gathered}
\end{equation}
The rationale behind these definitions will be explained for arbitrary multifold tensor products in Section \ref{Paragraph: Higher rank q-BI}.
By (\ref{comodule prop 1}) the operator \(\Gamma_{\{1,4\}}^q\) can equally be written as
\begin{equation}
\label{eq: Gamma 14 alternative expression}
	\Gamma_{\{1,4\}}^q = (1\otimes1\otimes\tau)(1\otimes\tau)\Delta(\Gamma^q),
\end{equation}
whereas due to the coassociativity (\ref{Coassociativity}), \(\Gamma_{\{1,2,4\}}^q\) allows the alternative expression
\begin{equation}
\label{eq: Gamma 124 alternative expression}
	\Gamma_{\{1,2,4\}}^q = (\Delta\otimes1\otimes1)(1\otimes\tau)\Delta(\Gamma^q).
\end{equation}
Explicit expressions for these operators, obtained by direct computation using (\ref{def: Coproduct}), (\ref{eq:Delta(Gamma)}) and (\ref{def:tau isomorphism}), can be found in Appendix A.

The definitions (\ref{def:tau isomorphism}) and (\ref{def:q-Bannai-Ito rank 2}) are motivated by the following identities. By direct calculation, one can verify that the relation
\begin{equation}
\label{eq: q-BI relations rank 2}
\{\Gamma_A^q,\Gamma_B^q\}_q = \Gamma_{C}^q + (q^{1/2}+q^{-1/2})\left(\Gamma_{A\cap B}^q\Gamma_{A\cup B}^q + \Gamma_{A\setminus (A\cap B)}^q\Gamma_{B\setminus (A\cap B)}^q\right).
\end{equation}
holds for \((A,B,C)\) any cyclic permutation of
\begin{gather*}
(\{1,2\},\{2,3\},\{1,3\}), \quad  (\{2,3\},\{3,4\},\{2,4\}), \\ (\{1,3\},\{3,4\},\{1,4\}), \quad (\{1,2\},\{2,4\},\{1,4\}), \\
\begin{align*}
(\{1,2\},\{2,3,4\},\{1,3,4\}), \quad (\{1,2,3\},\{3,4\},\{1,2,4\}), \quad (\{1,2,3\},\{2,3,4\},\{1,4\}).
\end{align*}
\end{gather*}
These relations are the extensions of the rank 1 \(q\)-Bannai-Ito relations (\ref{eq: q-Bannai-Ito rank 1}) to the fourfold tensor product. Anticipating the results of the next subsection, we state that the operators \(\Gamma_A^q\) with \(A\subseteq\{1,2,3,4\}\) will generate an algebra, which allows an embedding of \(\mathcal{A}_3^q\) and hence will be denoted the rank 2 \(q\)-deformed Bannai-Ito algebra.

\subsection{The higher rank $q$-deformed Bannai-Ito algebra}
\label{Paragraph: Higher rank q-BI}
We will now consider the \(n\)-fold tensor product algebra \(\ospq^{\otimes n}\) for arbitary \(n\geq 3\), governed by the multifold analog of the standard product rule (\ref{eq: Standard product law}):
\[
(a_1\otimes\dots\otimes a_n)(b_1\otimes\dots\otimes b_n) = a_1b_1\otimes\dots\otimes a_nb_n.
\]
To each set \(A\subseteq \{1,2\dots,n\}\) we will associate an element \(\Gamma_A^q\) of \(\ospq^{\otimes n}\), constructed by applying to \(\Gamma^q\) an intricate sequence of extension morphisms and inserting the unit on the lowest and highest positions in the tensor product. More precisely, we define
\begin{equation}
\label{def: extension procedure part 1}
\Gamma_A^q = \underbrace{1\otimes\dots\otimes 1}_{\min(A)-1\ \mathrm{times}}\otimes \left(\overrightarrow{\prod_{k=\min(A)+1}^{\max(A)}}\tau_{k-1,k}^{A}\right)(\Gamma^q)\otimes \underbrace{1\otimes\dots\otimes 1}_{n-\max(A)\ \mathrm{times}}.
\end{equation}
The arrow indicates that the morphisms \(\tau_{k-1,k}^{A}\) should be applied in order of increasing \(k\). Their definition, which reveals the actual extension algorithm, depends on whether \(k-1\) and \(k\) are elements of the set \(A\):
\begin{equation}
\label{def: extension procedure part 2}
\tau_{k-1,k}^{A}=\left\{
\arraycolsep=1.4pt\def\arraystretch{1.5}\begin{array}{lll}
\underbrace{1\otimes\dots\otimes 1}_{k-\min(A)-1\ \mathrm{times}}\otimes\Delta & \qquad \mathrm{if}\ k-1\in A \ \mathrm{and} \ k\in A, \\
(\underbrace{1\otimes\dots\otimes 1}_{k-\min(A)\ \mathrm{times}}\otimes\tau)(\underbrace{1\otimes\dots\otimes 1}_{k-\min(A)-1\ \mathrm{times}}\otimes\Delta) & \qquad\mathrm{if}\ k-1\in A \ \mathrm{and} \ k\notin A,\\
\underbrace{1\otimes\dots\otimes 1}_{k-\min(A)-1\ \mathrm{times}}\otimes\Delta\otimes 1
&\qquad\mathrm{if}\ k-1\notin A \ \mathrm{and}\ k\notin A,\\
\mathrm{id} & \qquad\mathrm{if}\ k-1\notin A \ \mathrm{and} \ k\in A,
\end{array}
\right.
\end{equation}
where \(\mathrm{id}\) denotes the identity mapping on \(\ospq^{\otimes (k-\min(A)+1)}\). The idea behind this definition is as follows. When two consecutive indices are present in the set \(A\), then the extension with respect to these indices is done by means of the coproduct \(\Delta\). Consider now the case when instead a \emph{hole} is present in the set \(A\), i.e. one of two consecutive indices \(k-1\) and \(k\) is an element of \(A\), the other is not. Then this hole must first be created by applying \( (\underbrace{1\otimes\dots\otimes1}_{k-\min(A) \ \textrm{times}}\otimes\tau)(\underbrace{1\otimes\dots\otimes1}_{k-\min(A)-1 \ \textrm{times}}\otimes\Delta)\). This amounts to creating first the term corresponding to the next element of \(A\) larger than \(k-1\) using \(\Delta\), and then inserting the hole at position \(k\) using \(\tau\). The hole may be enlarged using \(1\otimes\dots\otimes 1\otimes\Delta\otimes 1\) if necessary, i.e. if in the next step the index \(k+1\) is also not an element of \(A\). Note that upon creating the hole we have in fact applied two extension morphisms in one step. This will be compensated by applying the identity mapping the first time we encounter an index \(k'\) such that \(k'\in A\), \(k'-1\notin A\). This is where we \emph{close the hole}.

\begin{example}
We illustrate the extension algorithm for the case \(n = 7\) and \(A = \{2,5,6\}\). In this case
\[
\Gamma_{\{2,5,6\}}^q =1\otimes \left(\overrightarrow{\prod_{k=3}^{6}\tau_{k-1,k}^{\{2,5,6\}}}\right)(\Gamma^q)\otimes 1,
\]
where the extension morphisms \(\tau_{k-1,k}^{\{2,5,6\}}\) are defined as follows:

\begin{description}
	\item[\(k = 3\)] Since \(2\in A\) and \(3\notin A\), we have \(\tau_{2,3}^{\{2,5,6\}} = (1\otimes\tau)\Delta\). This amounts to creating the hole.
	\item[\(k = 4\)] Both 3 and 4 are not contained in \(A\), so we must enlarge the hole by applying \(1\otimes\Delta\otimes 1\).
	\item[\(k = 5\)] As \(4\notin A\) and \(5\in A\), we must close the hole at this point. Here \(\tau_{4,5}^{\{2,5,6\}} = \mathrm{id}\), i.e. doing nothing.
	\item[\(k = 6\)] Both 5 and 6 are elements of \(A\), hence we apply \(1\otimes1\otimes1\otimes\Delta\).
\end{description}
\end{example}

The properties of our extension morphisms allow us to find equivalent expressions for some of the operators \(\Gamma_A^q\). An example of such an expression that will be particularly useful is presented in the following lemma. Recall that by \([i;j]\) we denote the set \(\{i,i+1,\dots,j\}\).

\begin{lemma}
	\label{lemma: alternative expression one hole}
	For \(1< j\leq k\), one has:
	\[
	\Gamma_{[1;j-1]\cup\{k+2\}}^q = (\underbrace{1\otimes\dots\otimes1}_{k\ \mathrm{times}}\otimes\tau)\Gamma_{[1;j-1]\cup\{k+1\}}^q.
	\]
\end{lemma}
\begin{proof}
	We prove this by induction on \(k\). For \(k = j\) it follows from the extension procedure (\ref{def: extension procedure part 1})--(\ref{def: extension procedure part 2}) that
	\begin{align*}
	\Gamma_{[1;j-1]\cup\{j+2\}}^q & = (\underbrace{1\otimes\dots\otimes 1}_{j-1\ \mathrm{times}}\otimes\Delta\otimes 1)(\underbrace{1\otimes\dots\otimes 1}_{j-1\ \mathrm{times}}\otimes\tau)\Gamma_{[1;j]}^q \\
	& =  (\underbrace{1\otimes\dots\otimes 1}_{j\ \mathrm{times}}\otimes\tau)(\underbrace{1\otimes\dots\otimes 1}_{j-1\ \mathrm{times}}\otimes\tau)\Gamma_{[1;j]}^q \\
	& = (\underbrace{1\otimes\dots\otimes 1}_{j\ \mathrm{times}}\otimes\tau)\Gamma_{[1;j-1]\cup\{j+1\}}^q,
	\end{align*}
	where in the second line we have used (\ref{comodule prop 1}). Take now \(k > j\) and suppose the claim has been proven for \(k-1\), then we find
	\begin{align*}
	\Gamma_{[1;j-1]\cup\{k+2\}}^q & = (\underbrace{1\otimes\dots\otimes 1}_{k-1\ \mathrm{times}}\otimes\Delta\otimes 1)\Gamma_{[1;j-1]\cup\{k+1\}}^q \\
	& =  (\underbrace{1\otimes\dots\otimes 1}_{k-1\ \mathrm{times}}\otimes\Delta\otimes 1)(\underbrace{1\otimes\dots\otimes 1}_{k-1\ \mathrm{times}}\otimes\tau)\Gamma_{[1;j-1]\cup\{k\}}^q \\
	& = (\underbrace{1\otimes\dots\otimes 1}_{k\ \mathrm{times}}\otimes\tau)(\underbrace{1\otimes\dots\otimes 1}_{k-1\ \mathrm{times}}\otimes\tau)\Gamma_{[1;j-1]\cup\{k\}}^q \\
	& = (\underbrace{1\otimes\dots\otimes 1}_{k\ \mathrm{times}}\otimes\tau)\Gamma_{[1;j-1]\cup\{k+1\}}^q,
	\end{align*}
	where we have used the induction hypothesis on the second line and (\ref{comodule prop 1}) on the third line. This concludes the proof.
\end{proof}

Now that we have defined all operators \(\Gamma_A^q\), we are ready to state the definition of the higher rank \(q\)-deformed Bannai-Ito algebra.

\begin{definition}
	For any \(n\geq 3\), we denote by \(\mathcal{A}_n^q\) the subalgebra of \(\ospq^{\otimes n}\) with generators \(\Gamma_A^q\) with \(A\subseteq [n]\).
\end{definition}

This algebra can be identified as a higher rank generalization of the \(q\)-Bannai-Ito algebra \(\mathcal{A}_3^q\), as emerges from the following algebra relations. Let \(A\) and \(B\) be sets of integers between 1 and \(n\). We say that \(A\) \emph{matches} \(B\) if
\begin{equation}
\label{def: Matching sets}
\mathrm{max}(A\setminus(A\cap B)) < \mathrm{min}(A\cap B) \ \mathrm{and} \  \mathrm{max}(A\cap B) < \mathrm{min}(B\setminus(A\cap B)).
\end{equation}
An example of such matching sets is given by \(A = \{1,2,4,6\}, B = \{4,6,8\}\).

\begin{theorem}
	\label{prop: q-anticommutation Gamma}
	Let \(A, B \subseteq[n]\) be such that \(A\) matches \(B\) and \(A\) is a set of consecutive integers. Let \(C = (A\cup B) \setminus(A\cap B)\).
	Then the elements \(\Gamma_A^q\), \(\Gamma_B^q\) and \(\Gamma_C^q\) generate a rank 1 \(q\)-Bannai-Ito algebra:
	\begin{align}
	\label{eq: q-anticommutation Gamma}
	\begin{split}
	\{\Gamma_A^q,\Gamma_B^q\}_q  & = \Gamma_{C}^q + (q^{1/2}+q^{-1/2})\left(\Gamma_{A\cap B}^q\Gamma_{A\cup B}^q + \Gamma_{A\setminus( A\cap B)}^q\Gamma_{B\setminus (A\cap B)}^q\right), \\
	\{\Gamma_B^q,\Gamma_C^q\}_q  & = \Gamma_{A}^q + (q^{1/2}+q^{-1/2})\left(\Gamma_{B\cap C}^q\Gamma_{B\cup C}^q + \Gamma_{B\setminus (B\cap C)}^q\Gamma_{C\setminus (B\cap C)}^q\right), \\
	\{\Gamma_C^q,\Gamma_A^q\}_q  & = \Gamma_{B}^q + (q^{1/2}+q^{-1/2})\left(\Gamma_{C\cap A}^q\Gamma_{C\cup A}^q + \Gamma_{C\setminus (C\cap A)}^q\Gamma_{A\setminus (C\cap A)}^q\right).
	\end{split}
	\end{align}
\end{theorem}
\begin{proof}
	By the imposed requirements, the sets \(A\) and \(B\) are of the following form:
	\[
	A = [i;k], \quad B = [j;k]\cup \widetilde{B},
	\]	
	with \(i<j\leq k\) and all elements of \(\widetilde{B}\) strictly larger than \(k\). By (\ref{def: extension procedure part 1}) we may write \(i = 1\) without loss of generality.
	
	We will first consider the situation where \(j<k\) and \(\widetilde{B} = \{k+1\}\):
	\[
	A = [1;k], \quad B = [j;k+1].
	\]
	We apply the extension procedure (\ref{def: extension procedure part 1})--(\ref{def: extension procedure part 2}) to obtain the element \(\Gamma_A^q\):
	\begin{align*}
	\Gamma_A^q & = \left[(\underbrace{1\otimes\dots\otimes1}_{k-2 \ \mathrm{times}}\otimes\Delta)\dots(1\otimes\Delta)\Delta(\Gamma^q)\right]\otimes 1 \\
	& = (\underbrace{1\otimes\dots\otimes1}_{k-2 \ \mathrm{times}}\otimes\Delta\otimes 1)\dots(1\otimes\Delta\otimes 1)\left(\Delta(\Gamma^q)\otimes 1\right).
	\end{align*}
	We will now use the coassociativity (\ref{Coassociativity}) to rearrange the extension morphisms:
	\begin{equation}
	\label{eq: expression Gamma A final}
	\Gamma_A^q = (\Delta\otimes\underbrace{1\otimes\dots\otimes1}_{k-1 \ \mathrm{times}})\dots(\Delta\otimes\underbrace{1\otimes\dots\otimes1}_{k-j+2 \ \mathrm{times}})(\underbrace{1\otimes\dots\otimes1}_{k-j \ \mathrm{times}}\otimes\Delta\otimes 1)\dots(1\otimes\Delta\otimes 1)\left(\Gamma_{\{1,2\}}^q\right),
	\end{equation}
	where we have replaced \(\Delta(\Gamma^q)\otimes 1\) by the rank 1 operator \(\Gamma_{\{1,2\}}^q\) in the threefold tensor product.
	The motivation for this rearrangement becomes clear upon considering the expression for \(\Gamma_B^q\):
	\begin{align*}
	\Gamma_B^q & = \underbrace{1\otimes\dots\otimes 1}_{j-1 \ \mathrm{times}}\otimes \left[(\underbrace{1\otimes\dots\otimes1}_{k-j \ \mathrm{times}}\otimes\Delta)\dots(1\otimes\Delta)\Delta(\Gamma^q)\right] \\
	& = \underbrace{1\otimes\dots\otimes 1}_{j-1 \ \mathrm{times}}\otimes \left[(\underbrace{1\otimes\dots\otimes1}_{k-j-1 \ \mathrm{times}}\otimes\Delta\otimes 1)\dots(\Delta\otimes 1)\Delta(\Gamma^q)\right] \\
	& = \underbrace{1\otimes\dots\otimes 1}_{j-2 \ \mathrm{times}}\otimes \left[(\underbrace{1\otimes\dots\otimes1}_{k-j \ \mathrm{times}}\otimes\Delta\otimes 1)\dots(1\otimes\Delta\otimes 1)(1\otimes\Delta(\Gamma^q))\right].
	\end{align*}
	In the second line we have used the coassociativity to shift all morphisms \(\Delta\) over one position, in the last line we have taken a factor 1 inside the square brackets. Since \(\Delta(1) = 1 \otimes 1\) and since in the expression between square brackets, the term on the first position in the tensor product is 1, we may also produce the term \(\underbrace{1 \otimes\dots\otimes 1}_{j-2\ \mathrm{times}}\) by repeatedly applying \(\Delta\) to this expression:
	\begin{equation}
	\label{eq: expression Gamma B final}
	\Gamma_B^q = (\Delta\otimes\underbrace{1\otimes\dots\otimes1}_{k-1 \ \mathrm{times}})\dots(\Delta\otimes\underbrace{1\otimes\dots\otimes1}_{k-j+2 \ \mathrm{times}})(\underbrace{1\otimes\dots\otimes1}_{k-j \ \mathrm{times}}\otimes\Delta\otimes 1)\dots(1\otimes\Delta\otimes 1)\left(\Gamma_{\{2,3\}}^q\right).
	\end{equation}
	Note that we have replaced \(1\otimes\Delta(\Gamma^q)\) by the rank 1 operator \(\Gamma_{\{2,3\}}^q\).
	Observe that the sequences of applied extension morphisms in (\ref{eq: expression Gamma A final}) and (\ref{eq: expression Gamma B final}) are identical. The same sequence also arises when constructing \(\Gamma_C^q\), with \(C = (A\cup B)\setminus(A\cap B) = [1;j-1]\cup\{k+1\}\). Indeed, applying the extension procedure (\ref{def: extension procedure part 1})--(\ref{def: extension procedure part 2}) and using coassociativity, we find:
	\begin{align*}
	\Gamma_{[1;j-1] \cup \{k+1\}}^q = & \, (\underbrace{1\otimes\dots\otimes 1}_{k-2\ \mathrm{times}}\otimes\Delta\otimes 1)\dots(\underbrace{1\otimes\dots\otimes 1}_{j-1\ \mathrm{times}}\otimes\Delta\otimes 1)(\underbrace{1\otimes\dots\otimes 1}_{j-1 \ \mathrm{times}}\otimes\tau) \\ & \, (\Delta\otimes\underbrace{1\otimes\dots\otimes 1}_{j-2\ \mathrm{times}})\dots(\Delta\otimes 1)\Delta(\Gamma^q).
	\end{align*}
	The elements on the second position in the tensor product expression for \(\Delta(\Gamma^q)\) are left unaltered by the \(j-2\) morphisms of the form \(\Delta\otimes1\otimes\dots\otimes 1\), the coaction \(\tau\) is the first to act on these elements. We may thus rearrange the order and write first the morphisms acting on this second position. The same idea led us to the expression (\ref{eq: Gamma 124 alternative expression}) for \(\Gamma_{\{1,2,4\}}^q\). We obtain the following:
	\begin{align}
	\label{eq: Gamma 13 extended}
	\begin{split}
	\Gamma_{[1;j-1] \cup \{k+1\}}^q = & \, (\Delta\otimes\underbrace{1\otimes\dots\otimes1}_{k-1 \ \mathrm{times}})\dots(\Delta\otimes\underbrace{1\otimes\dots\otimes1}_{k-j+2 \ \mathrm{times}})\\ &\,(\underbrace{1\otimes\dots\otimes1}_{k-j \ \mathrm{times}}\otimes\Delta\otimes 1)\dots(1\otimes\Delta\otimes 1)(1\otimes\tau)\Delta(\Gamma^q) \\ = & \, (\Delta\otimes\underbrace{1\otimes\dots\otimes1}_{k-1 \ \mathrm{times}})\dots(\Delta\otimes\underbrace{1\otimes\dots\otimes1}_{k-j+2 \ \mathrm{times}})\\ &\,(\underbrace{1\otimes\dots\otimes1}_{k-j \ \mathrm{times}}\otimes\Delta\otimes 1)\dots(1\otimes\Delta\otimes 1)\left(\Gamma_{\{1,3\}}^q\right),
	\end{split}
	\end{align}
	where as before we have changed the notation \((1\otimes\tau)\Delta(\Gamma^q)\) to \(\Gamma_{\{1,3\}}^q\) in the last line. In (\ref{eq: Gamma 13 extended}) we now recognize the same sequence of extension morphisms as applied in (\ref{eq: expression Gamma A final}) and (\ref{eq: expression Gamma B final}). Applying this sequence now to the rank 1 operators \(\Gamma_{\{i\}}^q\) and \(\Gamma_{\{1,2,3\}}^q\), we find:
	\begin{align}
	\label{eq: Four identities}
	\begin{split}
	(\Delta\otimes\underbrace{1\otimes\dots\otimes1}_{k-1 \ \mathrm{times}})\dots(1\otimes\Delta\otimes 1)\Gamma_{\{1\}}^q & = \Gamma_{[1;j-1]}^q, \\
	(\Delta\otimes\underbrace{1\otimes\dots\otimes1}_{k-1 \ \mathrm{times}})\dots(1\otimes\Delta\otimes 1)\Gamma_{\{2\}}^q & = \Gamma_{[j;k]}^q, \\
	(\Delta\otimes\underbrace{1\otimes\dots\otimes1}_{k-1 \ \mathrm{times}})\dots(1\otimes\Delta\otimes 1)\Gamma_{\{3\}}^q & = \Gamma_{\{k+1\}}^q, \\
	(\Delta\otimes\underbrace{1\otimes\dots\otimes1}_{k-1 \ \mathrm{times}})\dots(1\otimes\Delta\otimes 1)\Gamma_{\{1,2,3\}}^q & = \Gamma_{[1;k+1]}^q.
	\end{split}
	\end{align}
	This follows from coassociativity and from the fact that \(\Delta(1) = 1\otimes 1\). Using the linearity and multiplicativity of the extension morphisms, we can bring the sequence outside the \(q\)-anticommutator:
	\[
	\{\Gamma_A^q,\Gamma_B^q\}_q = (\Delta\otimes\underbrace{1\otimes\dots\otimes1}_{k-1 \ \mathrm{times}})\dots(1\otimes\Delta\otimes 1)\{\Gamma_{\{1,2\}}^q,\Gamma_{\{2,3\}}^q\}_q.
	\]
	The rank 1 \(q\)-Bannai-Ito relations assert that
	\[
	\{\Gamma_{\{1,2\}}^q, \Gamma_{\{2,3\}}^q\}_q = \Gamma_{\{1,3\}}^q + \left(q^{1/2}+q^{-1/2}\right)\left(\Gamma_{\{2\}}^q\Gamma_{\{1,2,3\}}^q + \Gamma_{\{1\}}^q\Gamma_{\{3\}}^q\right),
	\]
	which combined with (\ref{eq: Gamma 13 extended}) and (\ref{eq: Four identities}) leads to the anticipated relation:
	\begin{equation}
	\label{eq: q-anticommutation special subcase}
	\{\Gamma_{[1;k]}^q,\Gamma_{[j;k+1]}^q\}_q = \Gamma_{[1;j-1]\cup\{k+1\}}^q + \left(q^{1/2}+q^{-1/2}\right)\left(\Gamma_{[j;k]}^q\Gamma_{[1;k+1]}^q + \Gamma_{[1;j-1]}^q\Gamma_{\{k+1\}}^q \right). 
	\end{equation}
	The relations for \(\{\Gamma_B^q,\Gamma_C^q\}_q\) and \(\{\Gamma_C^q,\Gamma_A^q\}_q\) now follow similarly from the other rank 1 \(q\)-Bannai-Ito relations
	\begin{align*}
	\{\Gamma_{\{2,3\}}^q, \Gamma_{\{1,3\}}^q\}_q = \Gamma_{\{1,2\}}^q + \left(q^{1/2}+q^{-1/2}\right)\left(\Gamma_{\{3\}}^q\Gamma_{\{1,2,3\}}^q + \Gamma_{\{1\}}^q\Gamma_{\{2\}}^q\right), \\
	\{\Gamma_{\{1,3\}}^q, \Gamma_{\{1,2\}}^q\}_q = \Gamma_{\{2,3\}}^q + \left(q^{1/2}+q^{-1/2}\right)\left(\Gamma_{\{1\}}^q\Gamma_{\{1,2,3\}}^q + \Gamma_{\{2\}}^q\Gamma_{\{3\}}^q\right).
	\end{align*}
	
	The case \(\widetilde{B} = \{k+1\}\) and \(j = k\) follows analogously using the simpler sequence
	\[
	(\Delta\otimes\underbrace{1\otimes\dots\otimes1}_{k-1 \ \mathrm{times}})\dots(\Delta\otimes1\otimes 1).
	\]
	
	Let us now return to the general case:
	\[
	A = [1;k], \quad B = [j;k]\cup \widetilde{B},
	\]
	where \(\widetilde{B}\) is a nonempty set whose elements are all strictly larger than \(k\). The expression for \(\Gamma_B^q\) can be obtained from the one for the \(k\)-fold tensor product operator \(\Gamma_{[j;k]}^q\):
	\[
	\Gamma_B^q = (\underbrace{1\otimes\dots\otimes 1}_{\mathrm{max}(B)-3 \ \mathrm{times}}\otimes\,\alpha_{\mathrm{max}(B)-2})\dots(\underbrace{1\otimes\dots\otimes 1}_{k-2 \ \mathrm{times}}\otimes\,\alpha_{k-1})\Gamma_{[j;k]}^q,
	\]
	where each \(\alpha_i\) is either \(1\otimes\Delta\), \(1\otimes\tau\) or \(\Delta\otimes 1\). From the extension procedure (\ref{def: extension procedure part 1})--(\ref{def: extension procedure part 2}) it is clear that \(\alpha_{k-1} = 1\otimes\Delta\), irrespective of whether \(k+1\) is contained in \(B\) or not. As \((\underbrace{1\otimes\dots\otimes 1}_{k-1 \ \mathrm{times}}\otimes\,\Delta)\Gamma_{[j;k]}^q = \Gamma_{[j;k+1]}^q\), we may write:
	\[
	\Gamma_B^q = (\underbrace{1\otimes\dots\otimes 1}_{\mathrm{max}(B)-3 \ \mathrm{times}}\otimes\,\alpha_{\mathrm{max}(B)-2})\dots(\underbrace{1\otimes\dots\otimes 1}_{k-1\ \mathrm{times}}\otimes\,\alpha_{k})\Gamma_{[j;k+1]}^q.
	\]
	Consider now the rank \(k-1\) operator \(\Gamma_{[1;k]}^q\), this is an element in the \((k+1)\)-fold tensor product with a 1 at the last position. \(\Gamma_A^q\) can be obtained from this operator upon repeatedly adding \(\otimes 1\) at the back. As \(\Delta(1) = \tau(1) = 1\otimes1\) and as \(\alpha_k\) cannot be \(\Delta\otimes 1\) by (\ref{def: extension procedure part 2}), this is equivalent to writing:
	\[
	\Gamma_A^q = (\underbrace{1\otimes\dots\otimes 1}_{\mathrm{max}(B)-3 \ \mathrm{times}}\otimes\,\alpha_{\mathrm{max}(B)-2})\dots(\underbrace{1\otimes\dots\otimes 1}_{k-1\ \mathrm{times}}\otimes\,\alpha_{k})\Gamma_{[1;k]}^q.
	\]
	Proceeding as before, we bring the sequence of extension morphisms outside the \(q\)-anticommutator:
	\[
	\{\Gamma_A^q,\Gamma_B^q\}_q = (1\otimes\dots\otimes1\otimes\,\alpha_{\mathrm{max}(B)-2})\dots(1\otimes\dots\otimes1\otimes\,\alpha_{k})\{\Gamma_{[1;k]}^q,\Gamma_{[j;k+1]}^q\}_q.
	\]
	All occurring operators can be built using this sequence:
	\begin{align*}
	(1\otimes\dots\otimes1\otimes\,\alpha_{\mathrm{max}(B)-2})\dots(1\otimes\dots\otimes1\otimes\,\alpha_{k})\Gamma_{[1;j-1]\cup\{k+1\}}^q & = \Gamma_{[1;j-1]\cup\widetilde{B}}^q, \\
	(1\otimes\dots\otimes1\otimes\,\alpha_{\mathrm{max}(B)-2})\dots(1\otimes\dots\otimes1\otimes\,\alpha_{k})\Gamma_{[1;k+1]}^q & = \Gamma_{[1;k]\cup\widetilde{B}}^q,
	\end{align*}
	and so on. For \(\alpha_k = 1\otimes\Delta\) these equalities follow immediately from the extension procedure (\ref{def: extension procedure part 1})--(\ref{def: extension procedure part 2}). In case \(\alpha_k=1\otimes\tau\) the equalities follow from Lemma \ref{lemma: alternative expression one hole}.
	Combined with (\ref{eq: q-anticommutation special subcase}), this leads indeed to
	\[
	\{\Gamma_A^q,\Gamma_B^q\}_q = \Gamma_{[1;j-1]\cup\widetilde{B}}^q + \left(q^{1/2}+q^{-1/2}\right)\left(\Gamma_{[j;k]}^q\Gamma_{[1;k]\cup\widetilde{B}}^q + \Gamma_{[1;j-1]}^q\Gamma_{\widetilde{B}}^q\right).
	\]
	The proof for \(\{\Gamma_B^q,\Gamma_C^q\}_q\) and \(\{\Gamma_C^q,\Gamma_A^q\}_q\) follows along the same lines.
\end{proof}

\begin{remark}
	The condition that \(A\) be a set of consecutive integers can in fact be omitted in the statement of Theorem \ref{prop: q-anticommutation Gamma}. Moreover, our definition of matching sets is more restrictive than necessary for the relations (\ref{eq: q-anticommutation Gamma}) to hold. The proof however becomes more complicated in these more general cases, and requires an additional definition. We plan to report on these minimal conditions in the near future \cite{DeClercq-2019}.
\end{remark}

Upon taking the limit \(q\rightarrow 1\), the relations (\ref{eq: q-anticommutation Gamma}) reduce to identities in the rank \(n-2\) Bannai-Ito algebra, introduced in \cite{DBAdv}. The claim that \(\mathcal{A}_n^q\) also has rank \(n-2\) is confirmed by the following proposition.

\begin{proposition}
	\label{prop: A contained in B}
	For \(A,B\subseteq [n]\) sets of consecutive integers such that \(A\subseteq B\), one has
	\[
	[\Gamma_A^q,\Gamma_B^q] = 0.
	\]
\end{proposition}
\begin{proof}
	The requirements on the sets \(A\) and \(B\) can be translated to
	\[
	A = [i;j], B = [1;k],
	\]
	with \(1\leq i \leq j \leq k\), and where by (\ref{def: extension procedure part 1}) we have \(\mathrm{min}(B) = 1\) without loss of generality. We proceed as in the proof of Theorem \ref{prop: q-anticommutation Gamma}: we rewrite \(\Gamma_A^q\) and \(\Gamma_B^q\) using a common sequence of extension morphisms and then bring this sequence outside the commutator. In case all inequalities are strict, i.e. \(i<j<k\), the sequence under consideration will be
	\begin{align*}
	&(\underbrace{1\otimes\dots\otimes1}_{k-2 \ \mathrm{times}}\otimes\Delta)\dots(\underbrace{1\otimes\dots\otimes1}_{j \ \mathrm{times}}\otimes\Delta)(\underbrace{1\otimes\dots\otimes1}_{j-2 \ \mathrm{times}}\otimes\Delta\otimes 1)\dots(\underbrace{1\otimes\dots\otimes1}_{i-1 \ \mathrm{times}}\otimes\Delta\otimes 1)\\&(\Delta\otimes\underbrace{1\otimes\dots\otimes 1}_{i-1 \ \mathrm{times}})\dots(\Delta\otimes 1\otimes 1),
	\end{align*}
	this gives \(\Gamma_{[i;j]}^q\) when applied to \(1\otimes\Gamma^q\otimes 1\) and \(\Gamma_{[1;k]}^q\) when applied to \((1\otimes\Delta)\Delta(\Gamma^q)\). The result then follows from the fact that \([1\otimes\Gamma^q\otimes 1, (1\otimes\Delta)\Delta(\Gamma^q)] = 0\). The other cases use similar sequences.
\end{proof}

Consider now chains
\[
A_1\subset A_2\subset \dots\subset A_k
\]
of subsets of \([n]\), ordered by inclusion, such that the operators \(\Gamma_{A_i}^q\) are all mutually commutative and each set has \(1 < \vert A_i\vert < n\). The length of the longest such chain can be taken to be the rank of the algebra \(\mathcal{A}_n^q\).

An example of such a chain is provided by the sets \([2],[3],\dots,[n-1]\) and the corresponding operators
\[
\Gamma_{[2]}^q,\Gamma_{[3]}^q,\dots,\Gamma_{[n-1]}^q.
\]
The subalgebra generated by these elements is abelian by the previous proposition and clearly no elements can be added to the chain without losing the properties that all \(\Gamma_{A_i}^q\) commute and that \(1<\vert A_i\vert < n\). We may conclude that \(\mathcal{A}_n^q\) is indeed of rank \(n-2\).

\begin{remark}
	\label{remark: A B commuting}
	More cases can be identified where \(\Gamma_A^q\) and \(\Gamma_B^q\) will commute. A trivial example is the case \(\mathrm{max}(A) < \mathrm{min}(B)\). Here in the expression for \(\Gamma_A^q\) all the positions in the tensor product corresponding to elements of the set \(B\) will be \(1\) and vice versa, as follows from (\ref{def: extension procedure part 1}). 
\end{remark}

As an immediate consequence of Theorem \ref{prop: q-anticommutation Gamma} we see that the operators of the form \(\Gamma_A^q\), with \(A\) a set of consecutive integers, are sufficient to generate the entire algebra.

\begin{corollary}
	\label{cor: Generating set}
	The set of operators \(\Gamma_{[i;j]}^q\) with \(i\leq j\leq n\) is a generating set for \(\mathcal{A}_n^q\).
\end{corollary}
\begin{proof}
	Let \(A\) be an arbitrary subset of \([n]\). Then writing the elements in consecutive order, we obtain an expression for \(A\) as a disjoint union of discrete intervals:
	\[
	A = [i_1;j_1] \cup [i_2;j_2] \cup [i_3;j_3] \cup\dots\cup [i_m;j_m],
	\] 
	with \(j_k+1<i_{k+1}\). We will prove the claim by induction on \(m\). If \(m = 1\), then \(A\) is itself a set of consecutive integers, so there is nothing to prove. Suppose now that \(m>1\) and that the statement has been proven for \(m-1\). We define the sets \(B\) and \(C\) as follows:
	\[
	B = [i_1;i_2-1], \quad C = [j_1+1;j_2]\cup[i_3;j_3]\cup\dots\cup[i_m;j_m].
	\]
	Note that
	\begin{align*}
	B\cap C = [j_1+1;i_2-1], \quad & B\cup C = [i_1;j_2] \cup [i_3;j_3]\cup\dots\cup[i_m;j_m], \\ B\setminus(B\cap C) = [i_1;j_1], \quad & C\setminus(B\cap C) = [i_2;j_2]\cup\dots\cup[i_m;j_m].
	\end{align*}
	Observe that set \(B\) matches set \(C\). Applying Theorem \ref{prop: q-anticommutation Gamma}, we obtain:
	\[
	\Gamma_A^q = \{\Gamma_B^q,\Gamma_C^q\}_q - \left(q^{1/2}+q^{-1/2}\right)\left(\Gamma_{B\cap C}^q\Gamma_{B\cup C}^q + \Gamma_{B\setminus B\cap C}^q\Gamma_{C\setminus B\cap C}^q\right).
	\]
	Each of the sets occurring in the right-hand side has strictly fewer holes than \(A\). Applying the induction hypothesis, the right-hand side may be rewritten using solely operators of the form \(\Gamma_{[i;j]}^q\). This proves our claim.
\end{proof}

Theorem \ref{prop: q-anticommutation Gamma} in fact asserts the existence of several copies of the rank 1 \(q\)-Bannai-Ito algebra inside \(\mathcal{A}_n^q\). An example of such an algebra is the subalgebra generated by \(\Gamma_{[m]}^q\), \(\Gamma_{\{m,m+1\}}^q\) and \(\Gamma_{[1;m-1]\cup\{m+1\}}^q\), where \(m\) is any natural number between \(2\) and \(n-1\). Inside this algebra one can identify the element
\begin{align}
\begin{split}
\label{def: Casimir operator C}
C_m = & \, (q^{-1/2}-q^{3/2})\Gamma_{\{m,m+1\}}^q\Gamma_{[1;m-1]\cup\{m+1\}}^q\Gamma_{[m]}^q \\ & + q\left(\Gamma_{\{m,m+1\}}^q\right)^2+q^{-1}\left(\Gamma_{[1;m-1]\cup\{m+1\}}^q\right)^2+ q\left(\Gamma_{[m]}^q\right)^2 \\ & - (q^{-1/2}-q^{3/2})\left(\Gamma_{\{m\}}^q\Gamma_{\{m+1\}}^q+\Gamma_{[m-1]}^q\Gamma_{[m+1]}^q \right)\Gamma_{\{m,m+1\}}^q \\ & - (q^{-1/2}-q^{3/2})\left(\Gamma_{\{m\}}^q\Gamma_{[m-1]}^q+\Gamma_{\{m+1\}}^q\Gamma_{[m+1]}^q \right)\Gamma_{[m]}^q \\ & - (q^{1/2}-q^{-3/2})\left(\Gamma_{\{m+1\}}^q\Gamma_{[m-1]}^q+\Gamma_{\{m\}}^q\Gamma_{[m+1]}^q\right)\Gamma_{[1;m-1]\cup\{m+1\}}^q,
\end{split}
\end{align}
in analogy with equation (3.11) in \cite{Genest&Vinet&Zhedanov-2016}. The operator \(C_m\) turns out to play a special role.

\begin{lemma}
	\label{lemma: Casimir C_m}
	For \(m\in\{2,3,\dots,n-1\}\), the operator \(C_m\) is the Casimir element of the algebra generated by \(\Gamma_{[m]}^q\), \(\Gamma_{\{m,m+1\}}^q\) and \(\Gamma_{[1;m-1]\cup\{m+1\}}^q\). Moreover, it allows the following equivalent expression:
	\begin{align}
	\label{eq: Alternative expression Casimir}
	\begin{split}
	C_m = & \,\left(\Gamma_{[m-1]}^q\right)^2 + \left(\Gamma_{\{m\}}^q\right)^2 + \left(\Gamma_{\{m+1\}}^q\right)^2 + \left(\Gamma_{[m+1]}^q\right)^2 \\ & - (q-q^{-1})^2\Gamma_{\{m\}}^q\Gamma_{\{m+1\}}^q\Gamma_{[m-1]}^q\Gamma_{[m+1]}^q - \frac{q}{(1+q)^2}.
	\end{split}
	\end{align}
\end{lemma}
\begin{proof}
	For \(m= 2\) this can be checked by direct computation, as was done in \cite{Genest&Vinet&Zhedanov-2016}. For \(m > 2\) one needs the following observation. Consider the sequence of extension morphisms
	\[
	(\Delta\otimes\underbrace{1\otimes\dots\otimes 1}_{m-1 \ \mathrm{times}})(\Delta\otimes\underbrace{1\otimes\dots\otimes 1}_{m-2 \ \mathrm{times}})\dots(\Delta\otimes 1\otimes 1).
	\]
	By a reasoning similar to the proof of Theorem \ref{prop: q-anticommutation Gamma}, one finds that this sequence has the following action on the threefold tensor product:
	\begin{align*}
	& \Gamma_{\{1\}}^q \mapsto \Gamma_{[m-1]}^q, \quad \Gamma_{\{2\}}^q \mapsto \Gamma_{\{m\}}^q, \quad \Gamma_{\{3\}}^q \mapsto \Gamma_{\{m+1\}}^q, \quad \Gamma_{\{1,2,3\}}^q \mapsto \Gamma_{[m+1]}^q, \\
	& \Gamma_{\{1,2\}}^q \mapsto \Gamma_{[m]}^q, \quad \Gamma_{\{2,3\}}^q \mapsto \Gamma_{\{m,m+1\}}^q, \quad \Gamma_{\{1,3\}}^q \mapsto \Gamma_{[1;m-1]\cup\{m+1\}}^q.
	\end{align*}
	Hence the equality for general \(m\) follows from the case \(m=2\).
	
	All operators occurring in (\ref{eq: Alternative expression Casimir}) commute with \(\Gamma_{[m]}^q\), \(\Gamma_{[1;m-1]\cup\{m+1\}}^q\) and \(\Gamma_{\{m,m+1\}}^q\), as follows from Proposition \ref{prop: A contained in B} and the relation
	\[
	\Gamma_{[1;m-1]\cup\{m+1\}}^q = \{\Gamma_{[m]}^q,\Gamma_{\{m,m+1\}}^q\}_q - (q^{1/2}+q^{-1/2})\left(\Gamma_{\{m\}}^q\Gamma_{[m+1]}^q + \Gamma_{\{m+1\}}^q\Gamma_{[m-1]}^q\right).
	\]
	Hence \(C_m\) is indeed the sought Casimir element.
\end{proof}

The \(q\)-Bannai-Ito relations in the considered subalgebra moreover allow us to state the following identities, which will be relied on in Section \ref{Paragraph - Irreducibility}.

\begin{lemma}
	For \(m\in\{2,3,\dots,n-1\}\), the operators \(\Gamma_{[m]}^q\) and \(\Gamma_{\{m,m+1\}}^q\) satisfy the relations
	\begin{align}
	\label{eq: q-BI relation with 2 generators 1}
	\begin{split}
	& \left(\Gamma_{[m]}^q\right)^2 \Gamma_{\{m,m+1\}}^q + \Gamma_{\{m,m+1\}}^q\left(\Gamma_{[m]}^q\right)^2 + (q+q^{-1})\Gamma_{[m]}^q\Gamma_{\{m,m+1\}}^q\Gamma_{[m]}^q \\
	= & \ \Gamma_{\{m,m+1\}}^q + (q^{1/2}+q^{-1/2})\left(\Gamma_{[m-1]}^q\Gamma_{[m+1]}^q + \Gamma_{\{m\}}^q\Gamma_{\{m+1\}}^q \right) \\ & + (q^{1/2}+q^{-1/2})^2\left(\Gamma_{\{m\}}^q\Gamma_{[m+1]}^q + \Gamma_{[m-1]}^q\Gamma_{\{m+1\}}^q\right)\Gamma_{[m]}^q
	\end{split}
	\end{align}
	and
	\begin{align}
	\label{eq: q-BI relation with 2 generators 2}
	\begin{split}
	& \left(\Gamma_{\{m,m+1\}}^q\right)^2\Gamma_{[m]}^q + \Gamma_{[m]}^q\left(\Gamma_{\{m,m+1\}}^q\right)^2 + (q+q^{-1})\Gamma_{\{m,m+1\}}^q\Gamma_{[m]}^q\Gamma_{\{m,m+1\}}^q \\ = & \ \Gamma_{[m]}^q + (q^{1/2}+q^{-1/2})\left(\Gamma_{\{m+1\}}^q\Gamma_{[m+1]}^q + \Gamma_{\{m\}}^q\Gamma_{[m-1]}^q \right) \\ & + (q^{1/2}+q^{-1/2})^2\left(\Gamma_{\{m\}}^q\Gamma_{[m+1]}^q + \Gamma_{[m-1]}^q\Gamma_{\{m+1\}}^q\right) \Gamma_{\{m,m+1\}}^q .
	\end{split}
	\end{align}
\end{lemma}
\begin{proof}
	The relation (\ref{eq: q-BI relation with 2 generators 1}) expresses the nested \(q\)-anticommutator
	\[
	\{\{\Gamma_{[m]}^q,\Gamma_{\{m,m+1\}}^q\}_q,\Gamma_{[m]}^q\}_q,
	\]
	which can be expanded using the algebra relations (\ref{eq: q-anticommutation Gamma}). Expression (\ref{eq: q-BI relation with 2 generators 2}) follows similarly upon calculating
	\[
	\{\Gamma_{\{m,m+1\}}^q,\{\Gamma_{[m]}^q,\Gamma_{\{m,m+1\}}^q\}_q\}_q.
	\]
\end{proof}

Recall that the tridiagonal algebra \cite{Terwilliger-2003} is generated by two elements \(A\) and \(A^{\ast}\) subject to the so-called tridiagonal relations
\begin{align}
\label{tridiagonal relations}
\begin{split}
[A,\, A^2A^{\ast}+A^{\ast}A^2-\beta AA^{\ast}A-\gamma(AA^{\ast}+A^{\ast}A)-\rho A^{\ast}] & = 0, \\
[A^{\ast},\, {A^{\ast}}^2A+A{A^{\ast}}^2-\beta A^{\ast}AA^{\ast}-\gamma^{\ast}(A^{\ast}A+AA^{\ast})-\rho^{\ast}A] & = 0,
\end{split}
\end{align}
for certain parameters \(\beta\), \(\gamma\), \(\gamma^{\ast}\), \(\rho\) and \(\rho^{\ast}\). These relations are also satisfied in the subalgebra generated by \(\Gamma_{[m]}^q\), \(\Gamma_{\{m,m+1\}}^q\) and \(\Gamma_{[1;m-1]\cup\{m+1\}}^q\).

\begin{corollary}
	\label{cor tridiagonal relations in q-BI}
	For any \(m\in\{2,3,\dots,n-1\}\), the tridiagonal relations (\ref{tridiagonal relations}) are satisfied by \(A = \Gamma_{[m]}^q\) and \(A^{\ast} = \Gamma_{\{m,m+1\}}^q\), under the parametrization
	\[
	\beta = -(q+q^{-1}), \quad \gamma = \gamma^{\ast} = 0, \quad \rho = \rho^{\ast} = 1.
	\]
\end{corollary}
\begin{proof}
	If \(A\) is any of the sets \(\{m\}\), \(\{m+1\}\), \([m-1]\) and \([m+1]\), then \(\Gamma_A^q\) will commute with both \(\Gamma_{[m]}^q\) and \(\Gamma_{\{m,m+1\}}^q\) by Proposition \ref{prop: A contained in B}. Hence the statement follows from (\ref{eq: q-BI relation with 2 generators 1}) and (\ref{eq: q-BI relation with 2 generators 2}).
\end{proof}
This suggests that the subalgebra generated by \(\Gamma_{[m]}^q\), \(\Gamma_{\{m,m+1\}}^q\) and \(\Gamma_{[1;m-1]\cup\{m+1\}}^q\) can be considered a quotient of the tridiagonal algebra. Moreover, suppose \(V\) is a finite-dimensional and irreducible module for the considered subalgebra such that both \(\Gamma_{[m]}^q\) and \(\Gamma_{\{m,m+1\}}^q\) are diagonalizable over \(V\). Then \((\Gamma_{[m]}^q,\,\Gamma_{\{m,m+1\}}^q)\) acts on \(V\) as a tridiagonal pair by \cite[Theorem 3.10]{Terwilliger-2003}. In what follows, we will no longer focus on the action of this specific subalgebra, but rather find an explicit module of the described type for the full algebra \(\mathcal{A}_n^q\). The typical tridiagonal action on this module will also occur in Theorem \ref{prop: Three-term recurrence relation}.

Finally, we may also apply the extension procedure (\ref{def: extension procedure part 1})--(\ref{def: extension procedure part 2}) to the \(\ospq\)-generators. The only sets that lend themselves to this extension are the sets \([i;j]\) of consecutive integers. This limitation arises from the fact that the coaction \(\tau\) is only defined on the subalgebra \(\mathcal{I}\) of Proposition \ref{lemma: prop tau}. Let \(i\leq j\leq n\), then we define
\begin{align}
\label{def: A,K,P higher rank}
\begin{split}
A_{\pm}^{[i;j]} &= \underbrace{1\otimes\dots\otimes 1}_{i-1\ \mathrm{times}}\otimes \left(\overrightarrow{\prod_{k=i+1}^{j}}\tau_{k-1,k}^{[i;j]}\right)(A_{\pm})\otimes \underbrace{P\otimes\dots\otimes P}_{n-j\ \mathrm{times}}, \\
K^{[i;j]} &= \underbrace{1\otimes\dots\otimes 1}_{i-1\ \mathrm{times}}\otimes \left(\overrightarrow{\prod_{k=i+1}^{j}}\tau_{k-1,k}^{[i;j]}\right)(K)\otimes \underbrace{P\otimes\dots\otimes P}_{n-j\ \mathrm{times}}, \\
P^{[i;j]} &= \underbrace{1\otimes\dots\otimes 1}_{i-1\ \mathrm{times}}\otimes \left(\overrightarrow{\prod_{k=i+1}^{j}}\tau_{k-1,k}^{[i;j]}\right)(P)\otimes \underbrace{1\otimes\dots\otimes 1}_{n-j\ \mathrm{times}}, 
\end{split}
\end{align}
These elements essentially arise from the same extension process as the \(\Gamma_A^q\) in (\ref{def: extension procedure part 1})--(\ref{def: extension procedure part 2}), up to the factors 1 and \(P\) in the highest tensor product positions.
Using (\ref{def: Coproduct}) this can be written more explicitly as
\begin{align}
\label{def: A,K,P higher rank explicit}
\begin{split}
A_{\pm}^{[i;j]} = & \sum_{l=1}^{j-i+1}\underbrace{1\otimes\dots\otimes 1}_{i-1 \ \mathrm{times}}\otimes\bigotimes_{k=1}^{l-1}K^{-1}\otimes A_{\pm}\otimes\bigotimes_{k=l+1}^{j-i+1}KP\otimes\underbrace{P\otimes\dots\otimes P}_{n-j \ \mathrm{times}}, \\
K^{[i;j]} = & \underbrace{1\otimes\dots\otimes 1}_{i-1 \ \mathrm{times}}\otimes\underbrace{K\otimes \dots \otimes K}_{j-i+1 \ \mathrm{times}}\otimes\underbrace{P\otimes\dots\otimes P}_{n-j \ \mathrm{times}}, \\
P^{[i;j]} = & \underbrace{1\otimes\dots\otimes 1}_{i-1 \ \mathrm{times}}\otimes\underbrace{P\otimes \dots \otimes P}_{j-i+1 \ \mathrm{times}}\otimes\underbrace{1\otimes\dots\otimes 1}_{n-j \ \mathrm{times}}. 
\end{split}
\end{align}
For each set \([i;j]\) these operators satisfy the \(\ospq\)-relations (\ref{def: ospq with K}), as follows from (\ref{def: A,K,P higher rank}) and the fact that each \(\tau_{k-1,k}^{A}\) is linear and multiplicative.
One now immediately observes how the \(\Gamma_{[i;j]}^q\) are related to the above defined operators:
\begin{equation}
\label{eq: Gamma in terms of other elements}
\Gamma_{[i;j]}^q = \left(-A_+^{[i;j]}A_-^{[i;j]}+\frac{q^{-1/2}\left(K^{[i;j]}\right)^2-q^{1/2}\left(K^{[i;j]}\right)^{-2}}{q-q^{-1}}\right)P^{[i;j]}.
\end{equation}
Indeed, we know from (\ref{def: Gamma^q}) that 
\[
\Gamma_{[i;j]}^q = \underbrace{1\otimes\dots\otimes 1}_{i-1\ \mathrm{times}}\otimes\left( \overrightarrow{\prod_{k=i+1}^j}\tau_{k-1,k}^{[i;j]}\left(\left(-A_+A_-+\frac{q^{-1/2}K^2-q^{1/2}K^{-2}}{q-q^{-1}} \right)P\right)\right)\otimes \underbrace{1\otimes\dots\otimes 1}_{n-j\ \mathrm{times}}
\]
and moreover, \(\underbrace{P\otimes\dots\otimes P}_{n-j\ \mathrm{times}}\) squares to \(\underbrace{1\otimes\dots\otimes 1}_{n-j\ \mathrm{times}}\) and the morphisms \(\tau_{k-1,k}^{[i;j]}\) are linear and multiplicative.

\subsection{Connection with the Askey-Wilson algebra}
\label{Paragraph: Connection with the AW-algebra}

In this subsection we review the definition of the Askey-Wilson algebra \(AW(3)\) and its universal analog. We discuss how it is connected to the rank 1 \(q\)-Bannai-Ito algebra and explain how our results determine a higher rank extension of \(AW(3)\). Note that an alternative construction for the rank 2 Askey-Wilson algebra can be found in \cite{Post}.

The Askey-Wilson algebra or Zhedanov algebra was originally defined in \cite{Zhedanov-1991} as the algebra with three generators \(K_0,K_1\) and \(K_2\) subject to the relations
\begin{align*}
	&[K_0,K_1]_Q = K_2, \quad [K_1,K_2]_Q = BK_1+C_0K_0+D_0, \\ &[K_2,K_0]_Q = BK_0+C_1K_1+D_1,
\end{align*}
where \(B,C_0,C_1,D_0,D_1\) are five complex constants, and where we have used the \(Q\)-commutator 
\[
[A,B]_Q = Q\,AB-Q^{-1}\,BA.
\]
We will denote this algebra by \(AW(3)_Q\), where we explicitly mention the deformation parameter \(Q\) for reasons that will soon become clear. An alternative and more symmetrical presentation has been proposed by Terwilliger in \cite{Terwilliger-2011}. Absorbing the five structure constants into three new constants \(\alpha_0,\alpha_1\) and \(\alpha_2\) and applying linear transformations to the generators, one obtains the equivalent form
\begin{align}
	\label{def: Askey-Wilson symmetric}
	\begin{split}
		&\frac{[A_0,A_1]_Q}{Q^2-Q^{-2}}+A_2 = \frac{\alpha_2}{Q+Q^{-1}}, \quad \frac{[A_1,A_2]_Q}{Q^2-Q^{-2}}+A_0 = \frac{\alpha_0}{Q+Q^{-1}}, \\ &\frac{[A_2,A_0]_Q}{Q^2-Q^{-2}}+A_1 = \frac{\alpha_1}{Q+Q^{-1}}.
	\end{split}
\end{align}
Taking for the \(\alpha_i\) central elements rather than scalars, one obtains the universal Askey-Wilson algebra, which we will also denote by \(AW(3)_Q\).

Consider now the quantum algebra \(\mathcal{U}_Q(\mathfrak{sl}_2)\), which has generators \(E,F, K\) and \(K^{-1}\) and defining relations \cite{Kassel}
\begin{align}
	\label{def: Uqsl2}
	KK^{-1} = K^{-1}K = 1, \quad KE = Q^2 EK, \quad KF = Q^{-2} FK, \quad [E,F] = \frac{K-K^{-1}}{Q-Q^{-1}}.
\end{align}
The Casimir element is
\[
\Lambda = \left(Q-Q^{-1}\right)^2 EF + Q^{-1} K + QK^{-1}.
\]
This algebra is related to \(\mathfrak{sl}_Q(2)\), as defined in \cite{Equitable Presentation}, by a transformation \(Q\mapsto Q^{1/2}\), one could in fact write \(\Uqsl = \mathfrak{sl}_{Q^2}(2)\).

\(\Uqsl\) can be equipped with the structure of a Hopf algebra. We state here the definition of its coproduct
\begin{align}
	\label{def: Coproduct Uqsl2}
	\begin{split}
	\Delta(E) = E\otimes 1 + K\otimes E, &\quad \Delta(F) = F\otimes K^{-1} + 1\otimes F, \\ \Delta(K) = K\otimes K, &\quad \Delta(K^{-1}) = K^{- 1}\otimes K^{- 1},
	\end{split} 
\end{align}
its counit
\[
\epsilon(E) = \epsilon(F) = 0, \quad \epsilon(K) = 1,
\]
and its antipode
\[
S(E) = -K^{-1}E, \quad S(F) = -FK, \quad S(K) = K^{-1}.
\]
It was observed in \cite{Granovskii&Zhedanov-1993} that the Askey-Wilson algebra can be realized within the threefold tensor product of \(\Uqsl\). More recently, this has been extended to the universal Askey-Wilson algebra in \cite{Huang}. Indeed, defining
\begin{equation}
	\label{def: Lambda12}
	\begin{alignedat}{3}
		\Lambda_{\{1,2\}} &= \Delta(\Lambda)\otimes 1, \quad &\Lambda_{\{2,3\}} &= 1\otimes \Delta(\Lambda), \quad &\Lambda_{\{1,2,3\}} &= (1\otimes\Delta)\Delta(\Lambda), \\
		\Lambda_{\{1\}} &= \Lambda\otimes 1\otimes 1, \quad &\Lambda_{\{2\}} &= 1\otimes \Lambda\otimes 1, \quad &\Lambda_{\{3\}} &= 1\otimes 1\otimes \Lambda,
	\end{alignedat}
\end{equation}
and
\begin{equation}
	\label{eq: Lambda 13}
\Lambda_{\{1,3\}} = \frac{1}{Q+Q^{-1}}\left[\Lambda_{\{1\}}\Lambda_{\{3\}}+\Lambda_{\{2\}}\Lambda_{\{1,2,3\}} - \frac{1}{Q-Q^{-1}}[\Lambda_{\{1,2\}},\Lambda_{\{2,3\}}]_Q \right],
\end{equation}
one can check that the relations (\ref{def: Askey-Wilson symmetric}) are satisfied by
\begin{gather*}
	A_0 = \Lambda_{\{1,2\}}, \quad A_1 = \Lambda_{\{2,3\}}, \quad A_2 = \Lambda_{\{1,3\}}, \\
	\alpha_0 = \Lambda_{\{1\}}\Lambda_{\{2\}}+ \Lambda_{\{3\}}\Lambda_{\{1,2,3\}}, \quad \alpha_1 = \Lambda_{\{2\}}\Lambda_{\{3\}}+ \Lambda_{\{1\}}\Lambda_{\{1,2,3\}}, \\ \alpha_2 = \Lambda_{\{1\}}\Lambda_{\{3\}}+ \Lambda_{\{2\}}\Lambda_{\{1,2,3\}}.
\end{gather*}

The Askey-Wilson algebra \(AW(3)_Q\) has a deep connection with the \(q\)-Bannai-Ito algebra \(\mathcal{A}_3^q\) if the deformation parameters are related by
\begin{equation}
	\label{transformation of deformation parameter}
	Q = iq^{1/2}.
\end{equation}
To see this we will first embed the quantum algebra \(\mathcal{U}_Q(\mathfrak{sl}_2)\) into \(\ospq\).
Defining
\begin{equation}
	\label{eq: Uqsl2 ospq correspondence}
	\widetilde{E} = \frac{1+Q^{-2}}{Q-Q^{-1}}A_+, \quad \widetilde{F} = -iQ \, A_-P, \quad \widetilde{K} = K^2P, \quad \widetilde{K^{-1}} = K^{-2}P,
\end{equation}
one can easily check that
\[
\widetilde{K}\widetilde{K^{-1}} = \widetilde{K^{-1}}\widetilde{K} = 1, \quad \widetilde{K}\widetilde{E} = Q^2\widetilde{E}\widetilde{K}, \quad \widetilde{K}\widetilde{F} = Q^{-2}\widetilde{F}\widetilde{K}, \quad [\widetilde{E},\widetilde{F}] = \frac{\widetilde{K}-\widetilde{K^{-1}}}{Q-Q^{-1}},
\]
as required by the \(\mathcal{U}_Q(\mathfrak{sl}_2)\)-relations (\ref{def: Uqsl2}). Under this embedding, the Casimir operator \(\Lambda\) is transformed into a scalar multiple of the Casimir element \(\Gamma^q\) for \(\ospq\):
\[
\widetilde{\Lambda} = \left(Q-Q^{-1}\right)^2 \widetilde{E}\widetilde{F} + Q^{-1} \widetilde{K} + Q\widetilde{K^{-1}} = -i(q-q^{-1})\Gamma^q.
\]
The tensor product elements (\ref{def: Lambda12}) and (\ref{eq: Lambda 13}) can also be translated to \(\ospq^{\otimes 3}\) using (\ref{eq: Uqsl2 ospq correspondence}). The resulting \(\widetilde{\Lambda_A}\) still satisfy the Askey-Wilson relations (\ref{def: Askey-Wilson symmetric}). For the first relation this is expressed explicitly by
\begin{equation}
	\label{eq: q-BI relation with Uqsl2}
	\frac{Q\,\widetilde{\Lambda}_{\{1,2\}}\widetilde{\Lambda}_{\{2,3\}}-Q^{-1}\,\widetilde{\Lambda}_{\{2,3\}}\widetilde{\Lambda}_{\{1,2\}}}{Q^2-Q^{-2}} + \widetilde{\Lambda}_{\{1,3\}} = \frac{\widetilde{\Lambda}_{\{1\}}\widetilde{\Lambda}_{\{3\}}+\widetilde{\Lambda}_{\{2\}}\widetilde{\Lambda}_{\{1,2,3\}}}{Q+Q^{-1}}.
\end{equation}
Writing now
\begin{equation}
\label{correspondence Gamma Lambda}
\widetilde{\Gamma_A^q} = \frac{i}{q-q^{-1}}\widetilde{\Lambda_A},
\end{equation}
for every \(A\subseteq \{1,2,3\}\) and writing every occurrence of \(Q\) as \(iq^{1/2}\), (\ref{eq: q-BI relation with Uqsl2}) is transformed into
\begin{equation}
	\label{eq: q-BI new embedding}
	\{\widetilde{\Gamma_{\{1,2\}}^q},\widetilde{\Gamma_{\{2,3\}}^q\}_q} = \widetilde{\Gamma_{\{1,3\}}^q} + (q^{1/2}+q^{-1/2})\left(\widetilde{\Gamma_{\{1\}}^q}\widetilde{\Gamma_{\{3\}}^q} + \widetilde{\Gamma_{\{2\}}^q}\widetilde{\Gamma_{\{1,2,3\}}^q}\right),
\end{equation}
which is precisely the first \(q\)-Bannai-Ito relation (\ref{eq: First BI-relation}). Similar statements hold for the second and third relation. We may conclude that \(AW(3)_Q\) and \(\mathcal{A}_3^q\) are isomorphic.

In Section \ref{Paragraph: Higher rank q-BI}, we have extended the rank 1 \(q\)-Bannai-Ito algebra to arbitrary rank using the algorithm (\ref{def: extension procedure part 1})--(\ref{def: extension procedure part 2}). The established isomorphism suggests that the same procedure can be used to define a rank \(n-2\) Askey-Wilson algebra \(AW(n)_Q\).

However, the operators \(\widetilde{\Gamma_A^q}\) for \(\vert A\vert \in\{2,3\}\) do not coincide with the tensor product elements \(\Gamma_A^q\) defined in (\ref{def: Intermediate Casimirs})-(\ref{def: Total Casimir}). This is because the coproduct (\ref{def: Coproduct Uqsl2}) on \(\mathcal{U}_Q(\mathfrak{sl}_2)\) is not compatible with the one on \(\ospq\) given in (\ref{def: Coproduct}) under the correspondence (\ref{eq: Uqsl2 ospq correspondence}). The relations (\ref{eq: q-BI new embedding}) hence establish a different embedding of the \(q\)-Bannai-Ito algebra inside \(\ospq^{\otimes 3}\). To construct the \(AW(n)_Q\)-generators one will thus need to repeat the process established in Section \ref{Paragraph: Higher rank q-BI}. Let us start by considering the subalgebra \(\mathcal{J}\) of \(U_Q(\mathfrak{sl}_2)\) generated by the elements \(EK^{-1}\), \(F\), \(K^{-1}\) and \(\Lambda\). As an analog of Proposition \ref{lemma: prop tau}, one can show the following.

\begin{proposition}
	The algebra \(\mathcal{J}\) is a left coideal subalgebra of \(U_Q(\mathfrak{sl}_2)\) and a left \(U_Q(\mathfrak{sl}_2)\)-comodule with coaction \(\tau: \mathcal{J}\to U_Q(\mathfrak{sl}_2)\otimes\mathcal{J}\) defined by
	\begin{align*}
	\tau(EK^{-1}) & = K^{-1}\otimes EK^{-1}, \\
	\tau(F) & = K\otimes F - Q^{-3}(Q-Q^{-1})^2 F^2K \otimes EK^{-1} \\ & \phantom{=} + Q^{-1}(Q+Q^{-1}) FK \otimes K^{-1} - Q^{-1} FK\otimes \Lambda \\
	\tau(K^{-1}) & = 1\otimes K^{-1}-Q^{-1}(Q-Q^{-1})^2 F\otimes EK^{-1} \\
	\tau(\Lambda) & = 1\otimes \Lambda.
	\end{align*}
\end{proposition}
\begin{proof}
	Along the same lines as Proposition \ref{lemma: prop tau}.
\end{proof}

Now one can copy (\ref{def: extension procedure part 1}) to define \(\Lambda_A \in \Uqsl^{\otimes n}\) for every \(A\subseteq [n]\), i.e.
\begin{equation}
\Lambda_A =  \underbrace{1\otimes\dots\otimes 1}_{\min(A)-1\ \mathrm{times}}\otimes \left(\overrightarrow{\prod_{k=\min(A)+1}^{\max(A)}}\tau_{k-1,k}^{A}\right)(\Lambda)\otimes \underbrace{1\otimes\dots\otimes 1}_{n-\max(A)\ \mathrm{times}},
\end{equation}
with \(\tau_{k-1,k}^{A}\) as in (\ref{def: extension procedure part 2}). This leads us to state the following.
\begin{definition}
For any \(n\geq 3\), the Askey-Wilson algebra \(AW(n)_Q\) of rank \(n-2\) is the subalgebra of \(U_Q(\mathfrak{sl}_2)^{\otimes n}\) generated by the elements \(\Lambda_A\) with \(A\subseteq [n]\).	
\end{definition}
By Theorem \ref{prop: q-anticommutation Gamma} these elements satisfy the relations
\begin{align*}
	\frac{[\Lambda_A,\Lambda_B]_Q}{Q^2-Q^{-2}}+\Lambda_{C} & = \frac{1}{Q+Q^{-1}}\left(\Lambda_{A\cap B}\Lambda_{A\cup B} + \Lambda_{A\setminus(A\cap B)}\Lambda_{B\setminus(A\cap B)}\right), \\
	\frac{[\Lambda_B,\Lambda_C]_Q}{Q^2-Q^{-2}}+\Lambda_{A} & = \frac{1}{Q+Q^{-1}}\left(\Lambda_{B\cap C}\Lambda_{B\cup C} + \Lambda_{B\setminus(B\cap C)}\Lambda_{C\setminus(B\cap C)}\right), \\
	\frac{[\Lambda_C,\Lambda_A]_Q}{Q^2-Q^{-2}}+\Lambda_{B} & = \frac{1}{Q+Q^{-1}}\left(\Lambda_{C\cap A}\Lambda_{C\cup A} + \Lambda_{C\setminus(C\cap A)}\Lambda_{A\setminus(C\cap A)}\right),
\end{align*}
for \(A\) a set of consecutive integers that matches \(B\), as defined in (\ref{def: Matching sets}), and \(C = (A\cup B)\setminus(A\cap B)\). 

Finally, for any \(m\in \{2,3,\dots,n-1\}\) one can consider the subalgebra of \(AW(n)_Q\) generated by \(\Lambda_{[m]}\), \(\Lambda_{\{m,m+1\}}\) and \(\Lambda_{[1;m-1]\cup\{m+1\}}\). Inside this subalgebra one may identify the element
\begin{align*}
\Omega_m = & \ \frac{1}{(Q^2-Q^{-2})^2}\left[Q\,\Lambda_{\{m,m+1\}}\Lambda_{[1;m-1]\cup\{m+1\}}\Lambda_{[m]}\right. \\ & + Q^{2}\left(\Lambda_{\{m,m+1\}}\right)^2+Q^{-2}\left(\Lambda_{[1;m-1]\cup\{m+1\}}\right)^2+ Q^{2}\left(\Lambda_{[m]}\right)^2 \\ & - Q\left(\Lambda_{\{m\}}\Lambda_{\{m+1\}}+\Lambda_{[m-1]}\Lambda_{[m+1]} \right)\Lambda_{\{m,m+1\}} - Q\left(\Lambda_{\{m\}}\Lambda_{[m-1]}+\Lambda_{\{m+1\}}\Lambda_{[m+1]} \right)\Lambda_{[m]} \\ & \left.- Q^{-1}\left(\Lambda_{\{m+1\}}\Lambda_{[m-1]}+\Lambda_{\{m\}}\Lambda_{[m+1]}\right)\Lambda_{[1;m-1]\cup\{m+1\}}\right].
\end{align*}
In analogy to Lemma \ref{lemma: Casimir C_m} one shows that \(\Omega_m\) is a Casimir element for the considered subalgebra. Under the embedding (\ref{eq: Uqsl2 ospq correspondence}) and the correspondence (\ref{transformation of deformation parameter})--(\ref{correspondence Gamma Lambda}), \(\Omega_m\) transforms to the element \(C_m\) of \(\mathcal{A}_n^q\) defined in (\ref{def: Casimir operator C}). 
Note that for \(m = 2\) this element becomes a scalar multiple of the Casimir element for \(AW(3)_Q\) proposed by Terwilliger in \cite[Lemma 6.1]{Terwilliger-2011}, where moreover five alternative expressions for \(\Omega_3\) are given which easily extend to \(\Omega_m\). Furthermore, it is straightforward to translate the equation (\ref{eq: Alternative expression Casimir}) and the identities (\ref{eq: q-BI relation with 2 generators 1}) and (\ref{eq: q-BI relation with 2 generators 2}) to the Askey-Wilson setting, under the aforementioned correspondences. 

It is also worth noting that the tridiagonal relations (\ref{tridiagonal relations}) are satisfied by \(A = \Lambda_{[m]}\) and \(A^{\ast} = \Lambda_{\{m,m+1\}}\) with parametrization \(\beta = Q^2+Q^{-2}\), \(\gamma = \gamma^{\ast} = 0\) and \(\rho = \rho^{\ast} = -(Q^2-Q^{-2})^2\), i.e. we have the following identities
\begin{align*}
[\Lambda_{[m]}, \Lambda_{[m]}^2\Lambda_{\{m,m+1\}}+\Lambda_{\{m,m+1\}}\Lambda_{[m]}^2-\beta\Lambda_{[m]}\Lambda_{\{m,m+1\}}\Lambda_{[m]}-\rho\Lambda_{\{m,m+1\}}] & = 0, \\
[\Lambda_{\{m,m+1\}}, \Lambda_{\{m,m+1\}}^2\Lambda_{[m]}+\Lambda_{[m]}\Lambda_{\{m,m+1\}}^2-\beta\Lambda_{\{m,m+1\}}\Lambda_{[m]}\Lambda_{\{m,m+1\}}-\rho^{\ast}\Lambda_{[m]}] & = 0.
\end{align*}
This follows from the remarks above and forms an analog for Corollary \ref{cor tridiagonal relations in q-BI}.

\begin{remark}
To conclude, we remark that in \cite[Section 3]{Terwilliger-2011} two automorphisms of \(AW(3)_Q\) were identified which generate the modular group \(\mathrm{PSL}_2(\mathbb{Z})\). In our notation, these morphisms act as follows:
\begin{align*}
P: (\Lambda_{\{1,2\}},\, \Lambda_{\{2,3\}},\, \Lambda_{\{1,3\}},\, \alpha,\, \beta,\, \gamma) &\to (\Lambda_{\{2,3\}},\, \Lambda_{\{1,3\}},\, \Lambda_{\{1,2\}},\, \beta,\,\gamma,\,\alpha), \\
S: (\Lambda_{\{1,2\}},\, \Lambda_{\{2,3\}},\, \Lambda_{\{1,3\}},\, \alpha,\, \beta,\, \gamma) &\to (\Lambda_{\{2,3\}},\, \Lambda_{\{1,2\}},\, \Lambda_{\{1,3\}}+\frac{[\Lambda_{\{1,2\}},\Lambda_{\{2,3\}}]}{Q-Q^{-1}},\, \beta,\, \alpha,\, \gamma)
\end{align*}
where we have written \(\alpha\), \(\beta\) and \(\gamma\) to denote the central elements
\[
\alpha = \Lambda_{\{1\}}\Lambda_{\{2\}}+\Lambda_{\{3\}}\Lambda_{\{1,2,3\}}, \quad \beta = \Lambda_{\{2\}}\Lambda_{\{3\}}+\Lambda_{\{1\}}\Lambda_{\{1,2,3\}}, \quad \gamma = \Lambda_{\{1\}}\Lambda_{\{3\}}+\Lambda_{\{2\}}\Lambda_{\{1,2,3\}}.
\]
In this notation, it is manifest that \(P\) corresponds to a permutation \(1\mapsto 2\mapsto 3\mapsto 1\) of the indexing sets. It would be of interest to look for analogs of these automorphisms in \(AW(n)_Q\) for \(n \geq 4\). We plan to look into this in future work.
\end{remark}

\section{The $\mathbb{Z}_2^n$ $q$-Dirac-Dunkl model}
\label{Section: Z2n q-Dirac-Dunkl model}
An explicit realization of the algebra \(\ospq\)  in terms of operators acting on functions of argument \(x\) was proposed in \cite{Genest&Vinet&Zhedanov-2016}. This inspires us to state the following definitions.

Consider the space \(\mathcal{P}(\mathbb{R}^n)\) of complex-valued polynomials in \(n\) real variables. Let \(\mu_1,\dots,\mu_n>0\) denote positive real parameters and take \(\gamma_i = \mu_i+\frac12\). Let \(r_i\) be the reflection with respect to the hyperplane \(x_i = 0\), i.e. \(r_if(x_1,\dots,x_n) = f(x_1,\dots,-x_i,\dots,x_n)\).

We define the \(q\)-Dunkl operator associated to the reflection group \(\mathbb{Z}_2^n\) as 
\begin{equation}
\label{def: q-Dunkl}
D_i^q = \frac{q^{\mu_i}}{q-q^{-1}}\left(\frac{T_{q,i}-r_i}{x_i} \right)-\frac{q^{-\mu_i}}{q-q^{-1}}\left(\frac{T_{q,i}^{-1}-r_i}{x_i}\right),
\end{equation}
where \(T_{q,i}\) denotes the \(q\)-shift operator with respect to the variable \(x_i\), sending \(f(x_1,\dots,x_n)\) to \(f(x_1,\dots,qx_i,\dots,x_n)\). Note that \(D_i^q\) reduces to the familiar \(\mathbb{Z}_2^n\) Dunkl operator \(\partial_{x_i} + \frac{\mu_i}{x_i}(1-r_i)\), see e.g. \cite{DTAMS, DX}, in the limit \(q\to 1\). The \(q\)-Dunkl operators, like their non-\(q\)-deformed counterparts, mutually commute: \([D_i^q,D_j^q] = 0\) for all \(i,j\in[n]\).

The \(q\)-deformed Dirac-Dunkl operator \(D_{[n]}^q\) and the position operator \(X_{[n]}^q\) are introduced as
\begin{equation}
D_{[n]}^q = \sum_{i\in[n]} D_i^qR_{[n],i}^q, \qquad X_{[n]}^q = \sum_{i\in [n]} x_iR_{[n],i}^q,
\end{equation}
where
\begin{equation}
\label{def: R_n,i^q}
R_{[n],i}^q = \left(\prod_{\substack{j \in [n]\\j<i}} q^{-\frac{\gamma_j}{2}}\left(T_{q,j}\right)^{-1/2}\right)\left(\prod_{\substack{j \in [n]\\j>i}}q^{\frac{\gamma_j}{2}}\left(T_{q,j}\right)^{1/2}r_j\right).
\end{equation}

Note that \(D_{[n]}^q\) and \(X_{[n]}^q\) are precisely the realizations of the elements \(A_{-}^{[n]}\) and \(A_+^{[n]}\) from (\ref{def: A,K,P higher rank explicit}) respectively, under the correspondence
\begin{equation}
\label{Realization of ospq}
A_+ \rightarrow x, \quad A_- \rightarrow D^q, \quad K \rightarrow q^{\gamma/2}T_q^{1/2}, \quad P\rightarrow r,
\end{equation}
which realizes the algebra \(\ospq\) in terms of \(q\)-shift operators and reflections. The additional generator \(A_0\) is then mapped to the element
\[
A_0 \rightarrow x\partial_x + \gamma,
\]
which follows from the expressions \(q^{x\partial_x} = T_q\) and \(q^{\frac{A_0}{2}} = K\). Writing
\begin{equation}
T_{q,[n]}=\prod_{i\in [n]}T_{q,i}, \qquad \gamma_{[n]} = \sum_{i\in [n]}\gamma_i = \sum_{i\in [n]}\mu_i + \frac{n}{2},
\end{equation}
we see that the elements \(K^{[n]}\) and \(P^{[n]}\) are realized inside the \(\mathbb{Z}_2^n\) \(q\)-Dirac-Dunkl model by \(q^{\frac{\gamma_{[n]}}{2}}T_{q,[n]}^{1/2}\) and \(\prod_{i\in[n]}r_i\) respectively.

The same identifications can be made for other sets. In Section \ref{Paragraph: Higher rank q-BI} we have associated to each set \(A\) of integers between 1 and \(n\) the element \(\Gamma_A^q\), and to each set \([i;j]\) of consecutive integers an \(n\)-fold extension of the generators \(A_{\pm}\) and \(K\). These tensor product elements can now be translated to the \(\mathbb{Z}_2^n\) \(q\)-Dirac-Dunkl setting via (\ref{Realization of ospq}). The elements \(A_-^{[i;j]}, A_+^{[i;j]}\) and \(K^{[i;j]}\) are then realized as follows:
\begin{equation}
\label{def: D_A^q,X_A^q}
D_{[i;j]}^q = \sum_{k=i}^jD_k^qR_{[i;j],k}^q, \quad X_{[i;j]}^q = \sum_{k=i}^jx_kR_{[i;j],k}^q, \quad q^{\frac{\gamma_{[i;j]}}{2}}T_{q,[i;j]}^{1/2}=\prod_{k=i}^jq^{\frac{\gamma_k}{2}}T_{q,k}^{1/2},
\end{equation}
with
\begin{equation}
\label{def: R_i,j,k^q}
R_{[i;j],k}^q = \left(\prod_{l=i}^{k-1}q^{-\frac{\gamma_l}{2}}\left(T_{q,l}\right)^{-1/2}\right)\left(\prod_{l=k+1}^{j}q^{\frac{\gamma_l}{2}}\left(T_{q,l}\right)^{1/2}r_l\right)\left(\prod_{l=j+1}^nr_l\right).
\end{equation}

We will continue to write \(\Gamma_A^q\) for the realization of the operators defined in (\ref{def: extension procedure part 1}) under the correspondence (\ref{Realization of ospq}). In this context we will call \(\Gamma_A^q\) a spherical \(q\)-Dirac-Dunkl operator. These operators then generate an explicit realization of our rank \(n-2\) \(q\)-Bannai-Ito algebra \(\mathcal{A}_n^q\).

\begin{remark}
The \(\mathbb{Z}_2^n\) Dunkl operator \cite{DTAMS} is defined as
\[
T_i = \partial_{x_i} + \frac{\mu_i}{x_i}(1-r_i).
\]
In \cite[Section 5]{DBAdv} the scalar \(\mathbb{Z}_2^n\) Dirac-Dunkl model was introduced via the operators
\begin{equation}
\label{def: Dirac-Dunkl model q = 1}
D_A = \sum_{i\in A}T_iR_i, \quad X_A = \sum_{i\in A}x_iR_i, \quad \Gamma_A = \left(\frac{[D_A,X_A]-1}{2}\right)\prod_{i\in A}r_i,
\end{equation}
where \(R_i = \prod_{j=i+1}^n r_j\). In the limit \(q\rightarrow 1\), we have
\[
D_i^q \rightarrow T_i, \quad R_{[i;j],k}^q\rightarrow R_k. 
\]
Hence the operators \(D_{[i;j]}^q\) and \(X_{[i;j]}^q\) reduce to the scalar \(\mathbb{Z}_2^n\) Dirac-Dunkl operator \(D_{[i;j]}\) and position operator \(X_{[i;j]}\) and similarly \(\Gamma_A^q\) reduces to \(\Gamma_A\). From this point of view, the extension procedure (\ref{def: extension procedure part 1})-(\ref{def: extension procedure part 2}) brings in fact a tensor product underpinning to the rank \(n-2\) and \(q=1\) Bannai-Ito algebra, which had not been established before. Note that the complexity of the expression (\ref{def:tau isomorphism}) defining the coaction \(\tau\) is remarkably reduced in the limit \(q\rightarrow 1\). 
\end{remark} 

Observe that the operators \(\Gamma_{\{i\}}^q\) will act as scalars in this representation, namely
\[
\Gamma_{\{i\}}^q = [\mu_i]_q.
\]
Explicit expressions for \(\Gamma_A^q\) in the rank 2 case can be deduced from Appendix A upon making the identifications (\ref{Realization of ospq}). The expressions in arbitrary rank follow analogously from (\ref{def: extension procedure part 1}) and (\ref{def: extension procedure part 2}). For sets \(A = [i;j]\) we can propose a general expression.

\begin{proposition}
	For sets \([i;j]\) with \(i<j\leq n\), we can write
	\begin{align}
	\label{eq: Gamma consecutive sets}
	\begin{split}
	\Gamma_{[i;j]}^q = & \left[\sum_{\substack{\{k,l\}\subseteq [i;j] \\ k< l}} M_{kl} + \sum_{k=i}^j\left([\mu_k]_qr_k -\frac{q^{\mu_k}}{q-q^{-1}}T_{q,k}+\frac{q^{-\mu_k}}{q-q^{-1}}T_{q,k}^{-1} \right) \left(R_{[i;j],k}^q\right)^2
	\right. \\  &\  + \left.\frac{q^{\gamma_{[i;j]}-\frac12}}{q-q^{-1}}T_{q,[i;j]}-\frac{q^{-(\gamma_{[i;j]}-\frac12)}}{q-q^{-1}}T_{q,[i;j]}^{-1}\right]\prod_{k=i}^jr_k,
	\end{split}
	\end{align}
	where \(M_{kl} = (q^{-1/2}x_kD_l^q-q^{1/2}x_lD_k^q)R_{[i;j],k}^qR_{[i;j],l}^q\).
\end{proposition}
\begin{proof}
	This follows upon translating the identity (\ref{eq: Gamma in terms of other elements}) to the \(q\)-Dirac-Dunkl model using (\ref{Realization of ospq}).
\end{proof}

The next proposition reveals how the operators \(\Gamma_A^q\) are related to the \(q\)-Dirac-Dunkl operator and the position operator.

\begin{proposition}
	\label{prop: Joint symmetries}
	The operators \(\Gamma_A^q\) are joint symmetries of \(D_{[n]}^q\) and \(X_{[n]}^q\): for every set \(A\subseteq [n]\) we have
	\[
	[D_{[n]}^q,\Gamma_A^q] = 0, \qquad [X_{[n]}^q,\Gamma_A^q] = 0.
	\]
\end{proposition}
\begin{proof}
	We will show the relations for sets of the form \(A = [i;j]\). Corollary \ref{cor: Generating set} then implies that the statement also holds for general sets \(A\). Returning to the tensor product approach, we may write
	\begin{align*}
	\Gamma_{[i;j]}^q & = (\underbrace{1\otimes\dots\otimes 1}_{j-2 \ \mathrm{times}}\otimes \Delta\otimes\underbrace{1\otimes\dots\otimes 1}_{n-j\ \mathrm{times}})\dots(\underbrace{1\otimes\dots\otimes 1}_{i-1 \ \mathrm{times}}\otimes \Delta\otimes\underbrace{1\otimes\dots\otimes 1}_{n-j\ \mathrm{times}})\\& \phantom{=\ }(\underbrace{1\otimes\dots\otimes 1}_{i-1 \ \mathrm{times}}\otimes \Gamma^q\otimes\underbrace{1\otimes\dots\otimes 1}_{n-j\ \mathrm{times}}).
	\end{align*}
	Using the coassociativity (\ref{Coassociativity}) of the coproduct we can express \(A_{\pm}^{[n]}\) as
	\begin{align*}
	A_{\pm}^{[n]} & = (\underbrace{1\otimes\dots\otimes 1}_{j-2 \ \mathrm{times}}\otimes \Delta\otimes\underbrace{1\otimes\dots\otimes 1}_{n-j\ \mathrm{times}})\dots(\underbrace{1\otimes\dots\otimes 1}_{i-1 \ \mathrm{times}}\otimes \Delta\otimes\underbrace{1\otimes\dots\otimes 1}_{n-j\ \mathrm{times}})\\& \phantom{=\ }(\underbrace{1\otimes\dots\otimes 1}_{n-j+i-2 \ \mathrm{times}}\otimes \Delta)\dots(1\otimes\Delta)\Delta(A_\pm).
	\end{align*}
	Hence the extension morphisms of the form \(1\otimes\dots\otimes\Delta\otimes\dots\otimes1\) can be moved in front and the remaining commutator
	\[
	\Big[(\underbrace{1\otimes\dots\otimes 1}_{n-j+i-2 \ \mathrm{times}}\otimes \Delta)\dots(1\otimes\Delta)\Delta(A_\pm),(\underbrace{1\otimes\dots\otimes 1}_{i-1 \ \mathrm{times}}\otimes \Gamma^q\otimes\underbrace{1\otimes\dots\otimes 1}_{n-j\ \mathrm{times}}) \Big]
	\]
	trivially vanishes, since \(\Gamma^q\) is central in \(\ospq\). Translating the obtained identity
	\[
	[A_{\pm}^{[n]},\Gamma_{[i;j]}^q] = 0
	\] 
	to the \(\Zn\) \(q\)-Dirac-Dunkl model using (\ref{Realization of ospq}) then yields the anticipated relations.
\end{proof}

These results also apply to smaller sets \([j]\) with \(j<n\).

\begin{corollary}
	\label{cor: D_j^q commutes smaller sets}
	For \(j<n\) and \(A\subseteq [j]\) one has
	\[
	[D_{[j]}^q,\Gamma_A^q] = 0, \qquad [X_{[j]}^q, \Gamma_A^q] = 0.
	\]
\end{corollary}
\begin{proof}
	We will once more prove the result for sets \(A = [k;\ell]\) with \(\ell\leq j\) and extend it to general sets \(A\) using Corollary \ref{cor: Generating set}. From the definition one sees
	\begin{equation}
	\label{eq: D smaller sets}
	D_{[j]}^q = \left(D_{[n]}^q - \sum_{i=j+1}^{n}D_{i}^qR_{[n],i}^q\right)\prod_{m=j+1}^n q^{-\frac{\gamma_m}{2}}T_{q,m}^{-1/2}.
	\end{equation}
	The operators \(\Gamma_{[k;\ell]}^q\) commute with \(D_{[n]}^q\) by Proposition \ref{prop: Joint symmetries}, with \(D_i^q\) for \(i>j\) since these only act on the corresponding variable \(x_i\), and with all of the other occurring operators by (\ref{eq: Gamma consecutive sets}). Indeed, by (\ref{eq: Gamma consecutive sets}) the operator \(\Gamma_{[k;\ell]}^q\) contains, apart from products of reflections and \(q\)-shift operators, only terms of the form \(x_sD_t^q\), with \(s,t\leq \ell\leq j\). These commute with \(R_{[n],i}^q\) and \(T_{q,m}^{-1/2}\) for \(i,m>j\), since \(R_{[n],i}^q\) contains the product \(T_{q,s}^{-1/2}T_{q,t}^{-1/2}\) and
	\[
	[T_{q,s}^{-1/2}T_{q,t}^{-1/2},x_sD_t^q] = 0.
	\]
	A similar approach works for \(X_{[j]}^q\).
\end{proof}

As a consequence, the higher rank \(q\)-deformed Bannai-Ito algebra may be regarded as an algebra of symmetries of the \(\Zn\) \(q\)-Dirac-Dunkl model. This correspondence will be further investigated in the next paragraphs.

\subsection{$q$-Dunkl monogenics}

Let \(\mathcal{P}_k(\mathbb{R}^n)\) be the space of complex-valued homogeneous polynomials of degree \(k\) in \(n\) real variables. The space \(\mathcal{M}_k^q(\mathbb{R}^n)\) of \(q\)-Dunkl monogenics of degree \(k\) with respect to the reflection group \(\mathbb{Z}_2^n\) is defined as
\begin{equation}
\mathcal{M}_k^q(\mathbb{R}^n) = \mathrm{Ker}(D_{[n]}^q) \cap \mathcal{P}_k(\mathbb{R}^n).
\end{equation}

The space \(\mathcal{P}_k(\mathbb{R}^n)\) allows a direct sum decomposition which can be taken as a \(q\)-Dunkl analog of the classical Fischer decomposition. We will need the following lemma.

\begin{lemma}
	\label{lemma commutation}
	For \(m\in \mathbb{N}\), one has
	\begin{align}
	\begin{split}
	\left[D_{[n]}^q,\left(X_{[n]}^q\right)^{2m}\right] = & \ q^{\gamma_{[n]}-\frac12}\frac{q^{2m}-1}{q-q^{-1}}  \left(X_{[n]}^q\right)^{2m-1}T_{q,[n]} \\
	& - q^{-\gamma_{[n]}+\frac12}\frac{q^{-2m}-1}{q-q^{-1}}\left(X_{[n]}^q\right)^{2m-1}T_{q,[n]}^{-1}.
	\end{split}
	\end{align}
	and
	\begin{align}
	\begin{split}
	\left\{D_{[n]}^q,\left(X_{[n]}^q\right)^{2m+1}\right\} = & \ q^{\gamma_{[n]}-\frac12}\frac{q^{2m+1}+1}{q-q^{-1}}\left(X_{[n]}^q\right)^{2m}T_{q,[n]} \\
	& - q^{-\gamma_{[n]}+\frac12}\frac{q^{-2m-1}+1}{q-q^{-1}}\left(X_{[n]}^q\right)^{2m}T_{q,[n]}^{-1}.
	\end{split}
	\end{align}
\end{lemma}
\begin{proof}
	A straightforward calculation using induction.
\end{proof}

\begin{proposition}
	The space \(\mathcal{P}_k(\mathbb{R}^n)\) has the direct sum decomposition
	\begin{equation}
	\label{eq: Fischer decomposition}
	\mathcal{P}_k(\mathbb{R}^n) = \bigoplus_{i=0}^k \left(X_{[n]}^q\right)^{i}\mathcal{M}_{k-i}^q(\mathbb{R}^n).
	\end{equation}
\end{proposition}
\begin{proof}
	We prove this by induction on \(k\). For \(k=0\) this is trivial, since both \(\mathcal{P}_0(\mathbb{R}^n)\) and \(\mathcal{M}_0^q(\mathbb{R}^n)\) are equal to the field of scalars \(\mathbb{C}\).
	
	Suppose the statement has been proven for \(k\), then we prove that \(\mathcal{P}_{k+1}(\mathbb{R}^n) = \mathcal{M}_{k+1}^q(\mathbb{R}^n) \oplus X_{[n]}^q\mathcal{P}_k(\mathbb{R}^n)\). Clearly, both \(\mathcal{M}_{k+1}^q(\mathbb{R}^n)\) and \(X_{[n]}^q\mathcal{P}_k(\mathbb{R}^n)\) are subspaces of \(\mathcal{P}_{k+1}(\mathbb{R}^n)\). Let \(\psi\in \mathcal{P}_{k+1}(\mathbb{R}^n)\) and \(\phi \in \mathcal{P}_k(\mathbb{R}^n)\). Then also \(D_{[n]}^q\psi\in\mathcal{P}_k(\mathbb{R}^n)\), hence by the induction hypothesis \(\phi\) and \(D_{[n]}^q\psi\) can be written in a unique way as
	\begin{align*}
	\begin{split}
	\phi =  & \ \chi_{k} + X_{[n]}^q\chi_{k-1}+\dots+\left(X_{[n]}^q\right)^{k}\chi_0, \\
	D_{[n]}^q\psi = & \ \xi_{k} + X_{[n]}^q\xi_{k-1}+\dots+\left(X_{[n]}^q\right)^{k}\xi_0,
	\end{split}
	\end{align*}
	with \(\chi_i,\xi_i\in \mathcal{M}_i^q(\mathbb{R}^n)\). Obviously \(\psi = (\psi-X_{[n]}^q\phi) + X_{[n]}^q\phi\), with
	\begin{equation}
	\label{requirement psi}
	\begin{array}{ll}
	&\psi-X_{[n]}^q\phi \in \mathcal{M}_{k+1}^q(\mathbb{R}^n) \\ \Leftrightarrow & D_{[n]}^q\psi = D_{[n]}^qX_{[n]}^q\chi_k + D_{[n]}^q\left(X_{[n]}^q\right)^2\chi_{k-1}+\dots+D_{[n]}^q\left(X_{[n]}^q\right)^{k+1}\chi_0.
	\end{array}
	\end{equation}
	By Lemma \ref{lemma commutation} and using the fact that \(D_{[n]}^q\chi_{i}=0\) and \(T_{q,[n]}\chi_{i} = q^{i}\chi_{i}\), we can write
	\[
	D_{[n]}^q\left(X_{[n]}^q\right)^{m}\chi_{k+1-m} = \alpha_{k+1,m}\left(X_{[n]}^q\right)^{m-1}\chi_{k+1-m},
	\]	
	with
	\[
	\alpha_{k+1,m} = q^{\gamma_{[n]}+k-m+\frac12}\frac{q^{m}-(-1)^m}{q-q^{-1}} - q^{-(\gamma_{[n]}+k-m+\frac12)}\frac{q^{-m}-(-1)^m}{q-q^{-1}}.
	\]
	Hence the requirement (\ref{requirement psi}) can only be met by taking 
	\[
	\chi_{m} = \frac{\xi_m}{\alpha_{k+1,k+1-m}}.
	\] 
	Note that \(\alpha_{k+1,k+1-m}\) is always non-zero since \(\vert q\vert \neq 1\).	We conclude that \(\psi\) can be written in a unique way as a sum of an element of \(\mathcal{M}_{k+1}^q(\mathbb{R}^n)\) and one of \(X_{[n]}^q\mathcal{P}_k(\mathbb{R}^n)\), which completes the proof.
\end{proof}

\subsection{The $\mathbf{CK}$-isomorphism}

An explicit isomorphism can be constructed between the space \(\mathcal{P}_k(\mathbb{R}^{j-1})\) of homogeneous polynomials in \(j-1\) variables and the space \(\mathcal{M}_k^q(\mathbb{R}^{j})\) of polynomial null-solutions of degree \(k\) of the \(q\)-Dirac-Dunkl operator \(D_{[j]}^q\).

\begin{proposition}
	\label{prop: CK extension}
	For \(j\in [n], j\geq 2\), the isomorphism \(\mathbf{CK}_{x_j}^{\mu_j}: \mathcal{P}_k(\mathbb{R}^{j-1}) \rightarrow \mathcal{M}_k^q(\mathbb{R}^{j})\) is given by
	\begin{equation}
	\label{eq: CK expression}
	\mathbf{CK}_{x_j}^{\mu_j} = \sum_{\alpha = 0}^{k} \frac{(-1)^{\frac{\alpha(\alpha+1)}{2}}q^{\frac{\alpha}{2}(\gamma_{[j]}+k-1)}}{[\mu_j,\alpha;q][\mu_j,\alpha-1;q]\dots[\mu_j,1;q]}x_j^{\alpha}\left(D_{[j-1]}^q\right)^{\alpha},
	\end{equation}
	with \([\mu_j,m;q] = [\mu_j+m]_q - (-1)^m[\mu_j]_q\).
\end{proposition}
\begin{proof}
	Let \(p\in\mathcal{P}_k(\mathbb{R}^{j-1})\). Anticipating the fact that \(\mathbf{CK}_{x_j}^{\mu_j}p\) should be an element of \(\mathcal{P}_k(\mathbb{R}^j) \), we can write
	\begin{equation}
	\label{eq: CK acting on p}
	\mathbf{CK}_{x_j}^{\mu_j}p(x_1,\dots,x_{j-1}) = \sum\limits_{\alpha = 0}^{k}x_j^{\alpha}p_{\alpha}(x_1,\dots,x_{j-1})
	\end{equation}
	for certain polynomials \(p_{\alpha} \in \mathcal{P}_{k-\alpha}(\mathbb{R}^{j-1}) \). We will now investigate the requirements imposed on the polynomials \(p_{\alpha}\) by the constraint that \(D_{[j]}^q\mathbf{CK}_{x_j}^{\mu_j}p = 0\). We will obtain expressions for \(p_{\alpha} \) in terms of \(p\), which then completely define \(\mathbf{CK}_{x_j}^{\mu_j}\).
	
	First observe that like in (\ref{eq: D smaller sets}) we have
	\begin{align*}
	D_{[j]}^q & = \sum_{i=0}^{j-1}D_i^qR_{[j],i}^q + D_j^qR_{[j],j}^q \\
	& = D_{[j-1]}^q\left(q^{\frac{\gamma_j}{2}}T_{q,j}^{1/2}\right) + D_j^qR_{[j],j}^q.
	\end{align*}
	From (\ref{def: q-Dunkl}) and (\ref{def: R_n,i^q}) one sees that
	\begin{align*}
	\begin{split}
	D_j^qR_{[j],j}^q(x_j^{\alpha}p_{\alpha}) & = q^{-\frac12(\gamma_{[j-1]}+k-\alpha)}\left(\frac{q^{\mu_j+\alpha}}{q-q^{-1}} - \frac{q^{-\mu_j-\alpha}}{q-q^{-1}}-(-1)^{\alpha}[\mu_j]_q \right)x_j^{\alpha-1}p_{\alpha} \\
	& = q^{-\frac12(\gamma_{[j-1]}+k-\alpha)}[\mu_j,\alpha;q]x_j^{\alpha-1}p_{\alpha},
	\end{split}
	\end{align*}
	where we have used the fact that \(p_{\alpha}\) is homogeneous of degree \(k-\alpha\). This results in the expression
	\begin{align*}
	\begin{split}
	D_{[j]}^q\mathbf{CK}_{x_j}^{\mu_j}p & = \sum_{\alpha=0}^{k} (-1)^{\alpha}q^{\frac{\alpha+\gamma_j}{2}}x_j^{\alpha}D_{[j-1]}^qp_{\alpha} + \sum_{\alpha = 1}^{k}q^{-\frac12(\gamma_{[j-1]}+k-\alpha)}[\mu_j,\alpha;q]x_j^{\alpha-1}p_{\alpha} \\
	& = \sum_{\alpha = 0}^{k} (-1)^{\alpha}q^{\frac{\alpha+\gamma_j}{2}}x_j^{\alpha}D_{[j-1]}^qp_{\alpha} + \sum_{\alpha = 0}^{k-1} q^{-\frac12(\gamma_{[j-1]}+k-\alpha-1)}[\mu_j,\alpha+1;q]x_j^{\alpha}p_{\alpha+1}.
	\end{split}
	\end{align*}
	As this sum must yield zero, we obtain by putting together all terms with an equal exponent of \(x_j\)
	\[
	p_{\alpha+1} = \frac{(-1)^{\alpha+1}q^{\frac12(\gamma_{[j]}+k-1)}}{[\mu_j,\alpha+1;q]}D_{[j-1]}^qp_{\alpha},
	\]
	for \(\alpha \in \{0,\dots,k-1\}\) and \(p_0 = p\), which follows from (\ref{eq: CK acting on p}) when evaluated in \(x_j = 0\). Solving this recursive relation yields
	\begin{equation}
	\label{eq: Recursion relation solved}
	p_{\alpha} = \frac{(-1)^{\frac{\alpha(\alpha+1)}{2}}q^{\frac{\alpha}{2}(\gamma_{[j]}+k-1)}}{[\mu_j,\alpha;q][\mu_j,\alpha-1;q]\dots[\mu_j,1;q]}\left(D_{[j-1]}^q\right)^{\alpha}p,
	\end{equation}
	hence (\ref{eq: CK expression}) follows. Note that \([\mu_j,i;q]\) is always non-zero since \(\vert q\vert \neq 1\).
	
	The constructed mapping is bijective and its inverse is given by evaluation at \(x_j=0\). Hence \(\mathbf{CK}_{x_j}^{\mu_j}\) is an isomorphism between \(\mathcal{P}_k(\mathbb{R}^{j-1})\) and \(\mathcal{M}_k^q(\mathbb{R}^{j})\).
\end{proof}

\begin{remark}
	In the limit \(q\rightarrow 1\) this reduces to the \(\mathbf{CK}\)-extension for the scalar \(\mathbb{Z}_2^n\) Dirac-Dunkl model, as defined in \cite[Section 5]{DBAdv}. This limit can be made explicit:
	\[
	\lim_{q\rightarrow 1}\left(\mathbf{CK}_{x_j}^{\mu_j}\right) = \Gamma(\mu_j+\tfrac12)\left(\widetilde{I}_{\mu_j-\frac12}(x_jD_{[j-1]})-\frac{x_jD_{[j-1]}}{2}\widetilde{I}_{\mu_j+\frac12}(x_jD_{[j-1]}) \right),
	\]
	where \(D_{[j-1]}\) is the \(q = 1\) scalar \(\Zn\) Dirac-Dunkl operator (\ref{def: Dirac-Dunkl model q = 1}), \(\Gamma\) is the Gamma function and \(\widetilde{I}_{\alpha}(x) = \left(\frac{2}{x}\right)^{\alpha}I_{\alpha}(x)\), with \(I_{\alpha}\) the modified Bessel function \cite{Special Functions}.
\end{remark}

We can now combine the \(\mathbf{CK}\)-extensions (\ref{eq: CK expression}) and the Fischer decomposition (\ref{eq: Fischer decomposition}) to obtain the following tower:
\begin{align}
\begin{split}
\label{eq: CK tower}
\mathcal{M}_k^q(\mathbb{R}^{n}) & = \mathbf{CK}_{x_n}^{\mu_n}\mathcal{P}_k(\mathbb{R}^{n-1})=\mathbf{CK}_{x_n}^{\mu_n}\left[\bigoplus\limits_{i=0}^{k}\left(X_{[n-1]}^q\right)^{k-i}\mathcal{M}_{i}^q(\mathbb{R}^{n-1}) \right] \\ &= \mathbf{CK}_{x_n}^{\mu_n}\left[\bigoplus\limits_{i=0}^{k}\left(X_{[n-1]}^q\right)^{k-i} \mathbf{CK}_{x_{n-1}}^{\mu_{n-1}} \mathcal{P}_i(\mathbb{R}^{n-2}) \right] \\ 
&= \mathbf{CK}_{x_n}^{\mu_n}\left[\bigoplus\limits_{i=0}^{k}\left(X_{[n-1]}^q\right)^{k-i}\mathbf{CK}_{x_{n-1}}^{\mu_{n-1}}\left[\bigoplus\limits_{j=0}^{i}\left(X_{[n-2]}^q\right)^{i-j}\mathcal{M}_{j}^q(\mathbb{R}^{n-2}) \right] \right] = \dots 
\end{split}
\end{align}
and so on, until one reaches \(\mathcal{P}_{m}(\mathbb{R})\), which is spanned by the function \(x_1^{m} \). This motivates the construction of the following basis functions for \(\mathcal{M}_k^q(\mathbb{R}^n)\).

\begin{definition}
	\label{def: Basis for M_k}
	Let \(\mathbf{j}=(j_1,j_2,\dots,j_{n-1})\in\mathbb{N}^{n-1} \), then we define the functions \(\psi_{\mathbf{j}} \) as:
	\begin{equation}
	\label{eq: Basis for M_k}
	\psi_{\mathbf{j}}(x_1,x_2,\dots,x_n) = \mathbf{CK}_{x_{n}}^{\mu_{n}}\left[\left(X_{[n-1]}^q\right)^{j_{n-1}}\mathbf{CK}_{x_{n-1}}^{\mu_{n-1}}\left[\dots \left[\left(X_{[2]}^q\right)^{j_{2}}\mathbf{CK}_{x_2}^{\mu_2}(x_1^{j_1})\right]\dots\right]\right].
	\end{equation}
	Then by (\ref{eq: CK tower}) the set \(\left\{\psi_{\mathbf{j}}: \mathbf{j}=(j_1,j_2,\dots,j_{n-1}) \in \mathbb{N}^{n-1} \ \mathrm{such\  that}\  \sum_{i = 1}^{n-1}j_i = k \right\}\) forms a basis for \(\mathcal{M}_k^q(\mathbb{R}^{n})\).
\end{definition}

\section{Action of the symmetry algebra $\mathcal{A}_n^q $}
\label{Section: Action of the symmetry algebra}
The operators \(\Gamma_A^q\) have a well-defined action on the space \(\mathcal{M}_k^q(\mathbb{R}^n)\) of \(q\)-Dunkl monogenics. To see this, first observe from the expression (\ref{eq: Gamma consecutive sets}) that each operator \(\Gamma_{[i;j]}^q\) preserves the degree of the homogeneous polynomial it acts on. This in fact holds for all operators \(\Gamma_A^q\), as they are generated by the operators \(\Gamma_{[i;j]}^q\) by Corollary \ref{cor: Generating set}. Moreover, if a polynomial \(\psi\) is a null-solution of the \(\Zn\) \(q\)-Dirac-Dunkl operator \(D_{[n]}^q\), then so is \(\Gamma_A^q\psi\), by Proposition \ref{prop: Joint symmetries}. In this section we study the action of the symmetry algebra \(\mathcal{A}_n^q\), generated by the operators \(\Gamma_A^q\), on the space \(\mathcal{M}_k^q(\mathbb{R}^n)\), and prove its irreducibility. 

\subsection{Spherical $q$-Dirac-Dunkl equation}
In this subsection we show that the operators \(\Gamma_{[\ell]}^q\) act diagonally on the basis for \(\mathcal{M}_k^q(\mathbb{R}^n)\) constructed in Definition \ref{def: Basis for M_k}. We start by considering the case \(\ell = n\).

\begin{lemma}
	\label{lemma first eigenvalue equation}
	Any \(\Psi_k \in \mathcal{M}_k^q(\mathbb{R}^n)\) is an eigenfunction of the operator \(\Gamma_{[n]}^q\):
	\begin{equation}
	\Gamma_{[n]}^q\Psi_k = (-1)^k \left[k+\gamma_{[n]}-\frac12\right]_q\Psi_k.
	\end{equation}
\end{lemma}
\begin{proof}
	In the \(\Zn\) \(q\)-Dirac-Dunkl model the identity (\ref{eq: Gamma in terms of other elements}) with \(i = 1\) and \(j = n\) becomes
	\[
	\Gamma_{[n]}^q = \left(-X_{[n]}^qD_{[n]}^q+\frac{q^{\gamma_{[n]}-\frac12}}{q-q^{-1}}T_{q,[n]}-\frac{q^{-\gamma_{[n]}+\frac12}}{q-q^{-1}}T_{q,[n]}^{-1}\right)\prod_{i=1}^nr_i.
	\]
	The anticipated expression now follows immediately from the fact that \(\Psi_k\) is homogeneous of degree \(k\) and a null-solution of \(D_{[n]}^q\). 
\end{proof}

In order to generalize this result to the operators \(\Gamma_{[\ell]}^q\) for \(\ell < n\) we will need a lemma characterizing how these commute with the \(\mathbf{CK}\)-isomorphisms.

\begin{lemma}
	\label{lemma CK commutes}
	For \(\ell<j\leq n\), one has
	\begin{equation}
	[\Gamma_{[\ell]}^q,\mathbf{CK}_{x_j}^{\mu_j}] = 0.
	\end{equation}
\end{lemma}
\begin{proof}
	First observe that the definition (\ref{eq: CK expression}) of \(\mathbf{CK}_{x_j}^{\mu_j}\) only involves powers of \(x_j\) and \(D_{[j-1]}^q\). Since \(j>\ell\), the only operators in the expression (\ref{eq: Gamma consecutive sets}) for \(\Gamma_{[\ell]}^q\) that act on \(x_j\) are of the form \(R_{[\ell],i}^qR_{[\ell],k}^q\) with \(i,k\leq \ell\). These will change the sign of \(x_j\) twice, thus leaving it unaltered, which ensures that \(\Gamma_{[\ell]}^q\) commutes with \(x_j\). The fact that \([\Gamma_{[\ell]}^q,D_{[j-1]}^q] = 0\) follows from Corollary \ref{cor: D_j^q commutes smaller sets}. The statement follows.
\end{proof}

This observation now allows us to introduce the spherical \(q\)-Dirac-Dunkl equation. For \(\mathbf{j} = (j_1,j_2,\dots,j_{n-1})\in\mathbb{N}^{n-1}\) and \(\ell<n\), we will denote by \(\mathbf{j}_{\ell}\) the truncated vector \((j_1,\dots,j_{\ell})\), and by \(\vert \mathbf{j}_{\ell}\vert\) the sum \(j_1+\dots+j_{\ell}\).

\begin{proposition}
	\label{prop: Spherical q-Dirac-Dunkl equation}
	Let \(\ell\in [n]\). The basis functions \(\psi_{\mathbf{j}}\) of \(\mathcal{M}_k^q(\mathbb{R}^n)\), introduced in Definition \ref{def: Basis for M_k}, satisfy the spherical \(q\)-Dirac-Dunkl equation 
	\begin{equation}
	\Gamma_{[\ell]}^q\psi_{\mathbf{j}} = \lambda_{\ell}(\mathbf{j})\psi_{\mathbf{j}},
	\end{equation}
	where the eigenvalue is given by
	\begin{equation}
	\label{def: eigenvalue}
	\lambda_{\ell}(\mathbf{j}) = (-1)^{\vert \mathbf{j}_{\ell-1}\vert} \left[\vert \mathbf{j}_{\ell-1}\vert + \gamma_{[\ell]} - \frac12 \right]_q.
	\end{equation}
\end{proposition}
\begin{proof}
	For \(\ell = n\) this is exactly the result of Lemma \ref{lemma first eigenvalue equation}. For \(\ell<n\), Lemma \ref{lemma CK commutes} and Corollary \ref{cor: D_j^q commutes smaller sets} tell us that \(\Gamma_{[\ell]}^q\) commutes with \(\mathbf{CK}_{x_j}^{\mu_j}\) for \(j>\ell\) and with \(X_{[j]}^m\) for \(j\geq \ell\) and arbitrary \(m\in\mathbb{N}\). As a consequence, we have
	\begin{align*}
	\begin{split}
	\Gamma_{[\ell]}^q\psi_{\mathbf{j}}  & = \Gamma_{[\ell]}^q \mathbf{CK}_{x_{n}}^{\mu_{n}}\left[\left(X_{[n-1]}^q\right)^{j_{n-1}}\mathbf{CK}_{x_{n-1}}^{\mu_{n-1}}\left[\dots \mathbf{CK}_{x_3}^{\mu_3}\left[\left(X_{[2]}^q\right)^{j_{2}}\mathbf{CK}_{x_2}^{\mu_2}(x_1^{j_1})\right]\dots\right]\right] \\
	& = \mathbf{CK}_{x_{n}}^{\mu_{n}}\left[\dots \left(X_{[\ell]}^q\right)^{j_{\ell}}\left[\Gamma_{[\ell]}^q\mathbf{CK}_{x_{\ell}}^{\mu_{\ell}}\left[\dots\mathbf{CK}_{x_3}^{\mu_3}\left[\left(X_{[2]}^q\right)^{j_2}\mathbf{CK}_{x_2}^{\mu_2}(x_1^{j_1}) \right]\dots\right]\right]\right].
	\end{split}
	\end{align*}
	As \(\mathbf{CK}_{x_{\ell}}^{\mu_{\ell}}\left[\left(X_{[\ell-1]}^q\right)^{j_{\ell-1}}\dots\mathbf{CK}_{x_2}^{\mu_2}(x_1^{j_1}) \right] \in \mathcal{M}_{\vert\mathbf{j}_{\ell-1}\vert}(\mathbb{R}^{\ell})\), we may conclude from Lemma \ref{lemma first eigenvalue equation} that 
	\[
	\Gamma_{[\ell]}^q\mathbf{CK}_{x_{\ell}}^{\mu_{\ell}}\left[\left(X_{[\ell-1]}^q\right)^{j_{\ell-1}}\dots\mathbf{CK}_{x_2}^{\mu_2}(x_1^{j_1}) \right] = \lambda_{\ell}(\mathbf{j})\mathbf{CK}_{x_{\ell}}^{\mu_{\ell}}\left[\left(X_{[\ell-1]}^q\right)^{j_{l-1}}\dots\mathbf{CK}_{x_2}^{\mu_2}(x_1^{j_1}) \right],
	\]
	so the result follows.
\end{proof}

We may conclude that the operators \(\Gamma_{[\ell]}^q\) with \(\ell\in[n]\) generate an abelian subalgebra of \(\mathcal{A}_n^q\) which is diagonalized by the considered basis. Note that the set of eigenvalues \(\{\lambda_{\ell}(\j): \ell\in \{2,3,\dots,n-1\}\}\) uniquely determines the vector \(\j\) and hence also the basis function \(\psi_{\j}\).

\subsection{Self-adjoint operators}
We can endow the space of polynomials with the following inner product. For \(P, Q\in\mathcal{P}(\mathbb{R}^n)\) we write
\begin{equation}
	\label{def: Inner product}
	\langle P, Q\rangle = \big(\overline{P(D_1^q,\dots,D_n^q)}Q(x_1,\dots,x_n)\big)\big\vert_{\mathbf{x} = 0}.
\end{equation}
One easily verifies that this defines an inner product under the condition that \(q\) be a positive real number different from 1. This is in fact the \(q\)-analog of the Dunkl-Fischer inner product, as introduced in \cite[Chapter 5.2]{DX}. For two monomials \(\mathbf{x}^{\mathbf{\alpha}},\mathbf{x}^{\mathbf{\beta}}\), with \(\mathbf{\alpha},\mathbf{\beta}\in\mathbb{N}^n\), one obtains
\[
\langle \mathbf{x}^{\mathbf{\alpha}},\mathbf{x}^{\mathbf{\beta}}\rangle = \delta_{\mathbf{\alpha},\mathbf{\beta}}\left(\prod_{i=1}^n\prod_{j=1}^{\alpha_i}[\mu_i,j;q]\right),
\]
with \([\mu_i,j;q]\) as defined in Proposition \ref{prop: CK extension}. Note that each \([\mu_i,j;q]\) is positive, since both \(\mu_i\) and \(q\) are positive as well. The adjoints of the operators \(x_i, D_i^q, T_{q,i}\) and \(r_i\) with respect to this inner product are readily checked to be
\begin{equation}
	\label{eq: Adjoints}
	x_i^\dagger = D_i^q, \quad {D_i^q}^{\dagger} = x_i, \quad T_{q,i}^{\dagger} = T_{q,i}, \quad r_i^{\dagger} = r_i.
\end{equation}
This enables us to compute the adjoints of the operators \(\Gamma_A^q\).

\begin{lemma}
	\label{lemma: Self-adjoint}
	The operators \(\Gamma_{[i;j]}^q\) with \(i\leq j\) are self-adjoint.
\end{lemma}
\begin{proof}
	This immediately follows from the expression (\ref{eq: Gamma consecutive sets}) for \(\Gamma_{[i;j]}^q\). Note that \(M_{kl}^{\dagger} = R_{[i;j],l}^qR_{[i;j],k}^q\left(q^{-1/2}x_lD_k^q-q^{1/2}x_kD_l^q\right)\), which equals \(M_{kl}\) by the definition (\ref{def: R_i,j,k^q}) of \(R_{[i;j],k}^q\).
\end{proof}

\subsection{Irreducibility of the action}
\label{Paragraph - Irreducibility}
The action of the symmetry algebra \(\mathcal{A}_n^q\) on the space \(\mathcal{M}_k^q(\mathbb{R}^n)\) turns out to be irreducible. To prove this statement, it will suffice to describe explicitly the action of its subalgebra generated by \(\Gamma_{[m]}^q\), \(\Gamma_{\{m,m+1\}}^q\) and \(\Gamma_{[1;m-1]\cup\{m+1\}}^q\), for arbitrary \(m\in\{2,3,\dots,n-1\}\). This is where the Casimir elements considered in Lemma \ref{lemma: Casimir C_m} will come in handy.

The action of the operator \(\Gamma_{\{m,m+1\}}^q\) on the basis functions \(\psi_{\mathbf{j}}\) takes a tridiagonal form. To describe this action, we will first introduce the terminology \emph{\(k\)-allowability} of a vector. We call a vector \(\j\in\mathbb{Z}^{n-1}\) \(k\)-allowable if all its entries are non-negative and add up to \(k\). We will also need the notation \(\mathbf{h}_m = (0,\dots,1,-1,\dots,0)\in\mathbb{N}^{n-1}\) for the vector which has a 1 at position \(m-1\), a \(-1\) at position \(m\) and zeros elsewhere. 

\begin{theorem}
	\label{prop: Three-term recurrence relation}
	Let \(m\in\{2,3,\dots,n-1\}\) and consider the basis functions \(\psi_{\mathbf{j}}\) as constructed in Definition \ref{def: Basis for M_k}. Then there exist scalars \(A_{\mathbf{j}}\), \(B_{\mathbf{j}}\) and \(C_{\mathbf{j}}\) such that the following holds:
	\begin{equation}
	\Gamma_{\{m,m+1\}}^q\psi_{\mathbf{j}} = B_{\j}\psi_{\j-\h}+A_{\mathbf{j}}\psi_{\j}+C_{\j}\psi_{\j+\h}.
	\end{equation}
	The constant \(B_{\j}\) is non-zero if the vectors \(\j\) and \(\j-\h\) are both \(k\)-allowable and \(C_{\j}\) is non-zero if \(\j\) and \(\j+\h\) are both \(k\)-allowable.
\end{theorem}
\begin{proof}
	Acting with \(\Gamma_{\{m,m+1\}}^q\) on the function \(\psi_{\j}\), we obtain once more a polynomial in \(\mathcal{M}_k^q(\mathbb{R}^n)\), which can thus be expanded in the considered basis:
	\begin{equation}
	\label{eq: Recurrence first expression}
	\Gamma_{\{m,m+1\}}^q\psi_{\j} = \sum_{\j'} A_{\j'}\psi_{\j'},
	\end{equation}
	where the sum runs over all \(k\)-allowable vectors \(\j'\). We now let both sides of expression (\ref{eq: q-BI relation with 2 generators 1}) act on the function \(\psi_{\j}\) and isolate the coefficient of \(\psi_{\j'}\). The resulting equation becomes:
	\begin{equation}
	\label{eq: Relation with deltas}
	A_{\mathbf{j}'}\left(\lambda_m(\mathbf{j})^2 + \lambda_m(\mathbf{j}')^2 + (q+q^{-1})\lambda_m(\mathbf{j})\lambda_m(\mathbf{j}') - 1 \right) = \delta_{\j,\j'}\alpha_{m}(\j),
	\end{equation}
	where \(\delta_{\j,\j'}\) stands for \(\delta_{j_1,j_1'}\delta_{j_2,j_2'}\dots\delta_{j_{n-1},j_{n-1}'}\) and where
	\begin{align*}
	\begin{split}
	\alpha_m(\j) = & \,(q^{1/2}+q^{-1/2})\left(\lambda_{m-1}(\j)\lambda_{m+1}(\j)+[\mu_m]_q[\mu_{m+1}]_q\right)\\ & +  (q^{1/2}+q^{-1/2})^2\left([\mu_m]_q\lambda_{m+1}(\mathbf{j}) + [\mu_{m+1}]_q\lambda_{m-1}(\j)\right)\lambda_{m}(\mathbf{j}).
	\end{split}
	\end{align*}
	
	For \(\j' = \j\), the equation (\ref{eq: Relation with deltas}) gives an expression for \(A_{\j}\):
	\[
	A_{\j} = \frac{\alpha_m(\j)}{(2+q+q^{-1})\lambda_m(\j)^2-1}.
	\]
	The denominator in this expression can be checked to be non-zero since \(\vert q\vert \neq 1\).
	
	If \(\j'\neq\j\), then the coefficient \(A_{\j'}\) will vanish unless
	\[
	\lambda_m(\mathbf{j})^2 + \lambda_m(\mathbf{j}')^2 + (q+q^{-1})\lambda_m(\mathbf{j})\lambda_m(\mathbf{j}') = 1,
	\]
	which is only fulfilled for \(\vert\j'_{m-1}\vert = \vert\j_{m-1}\vert\pm1\). As before, \(\vert\j_{m-1}\vert\) denotes the sum \(j_1+\dots+j_{m-1}\). As a result, the equation (\ref{eq: Recurrence first expression}) reduces to
	\begin{equation}
	\label{eq: Recurrence second expression}
	\Gamma_{\{m,m+1\}}^q\psi_{\j} = A_{\j}\psi_{\j} + \sum_{ \vert\j_{m-1}'\vert = \vert\j_{m-1}\vert-1}A_{\j'}\psi_{\j'} + \sum_{ \vert\j_{m-1}'\vert = \vert\j_{m-1}\vert+1}A_{\j'}\psi_{\j'}.
	\end{equation}
	We can now show that the coefficients \(A_{\j'}\) vanish unless \(\j' = \j\pm\h\). The proof requires several steps, which should be taken in consecutive order. For \(m = 2\), one should skip the indicated Step 1 and immediately jump to Step \(m-1\).
	\begin{itemize}
		\item Step 1: Let \(\Gamma_{[2]}^q\) act on both sides of equation (\ref{eq: Recurrence second expression}). As this step is only executed if \(m\geq 3\), we know from Remark \ref{remark: A B commuting} that \(\Gamma_{[2]}^q\) commutes with \(\Gamma_{\{m,m+1\}}^q\), hence we obtain from Proposition \ref{prop: Spherical q-Dirac-Dunkl equation}
		\begin{align*}
		\begin{split}
		& \lambda_2(\j)\Gamma_{\{m,m+1\}}^q\psi_{\j} \\ = & \, \lambda_2(\j)A_{\j}\psi_{\j} + \sum_{ \vert\j_{m-1}'\vert = \vert\j_{m-1}\vert-1}A_{\j'}\lambda_2(\j')\psi_{\j'} + \sum_{ \vert\j_{m-1}'\vert = \vert\j_{m-1}\vert+1}A_{\j'}\lambda_2(\j')\psi_{\j'}.
		\end{split}
		\end{align*}
		At the same time, this yields the right hand side of (\ref{eq: Recurrence second expression}) multiplied by \(\lambda_2(\j)\). By the linear independence of the functions \(\psi_{\j'}\) we obtain that \(A_{\j'}\) vanishes unless \(\lambda_2(\j') = \lambda_2(\j)\), or equivalently
		\[
		j_1' = j_1.
		\]
		\item Step 2 to \(m-2\): In a similar fashion, acting on (\ref{eq: Recurrence second expression}) with \(\Gamma_{[3]}^q,\Gamma_{[4]}^q,\dots,\Gamma_{[m-1]}^q\), we find that
		\[
		j_2' = j_2, \dots, j_{m-2}' = j_{m-2}.
		\]
		\item Step \(m-1\): Act on both sides of (\ref{eq: Recurrence second expression}) with \(\Gamma_{[m+1]}^q\), which once more commutes with \(\Gamma_{\{m,m+1\}}^q\) by Proposition \ref{prop: A contained in B}.
		Proceeding as before, we find that \(A_{\j'}\) vanishes unless \(\lambda_{m+1}(\j') = \lambda_{m+1}(\j)\), or equivalently
		\[
		\vert\j_{m}'\vert = \vert\j_{m}\vert.		
		\]
		Together with the results obtained in the preceding steps, this yields
		\[
		j_{m-1}'+j_{m}' = j_{m-1}+j_m.
		\]
		Combining this with the constraint that \(\vert\j'_{m-1}\vert = \vert\j_{m-1}\vert\pm1\) we conclude that
		\[
		j_{m-1}' = j_{m-1}\pm 1, \quad j_m' = j_{m}\mp 1.
		\]
		\item Step \(m\) to \(n-3\): Acting on (\ref{eq: Recurrence second expression}) with \(\Gamma_{[m+2]}^q,\dots,\Gamma_{[n-1]}^q\), we conclude that
		\[
		j_{m+1}' = j_{m+1},\dots,j_{n-2}' = j_{n-2}.
		\]
		As we only consider \(k\)-allowable vectors, whose entries add up to \(k\), this also implies that \(j_{n-1}' = j_{n-1}\).
	\end{itemize}
	We conclude that indeed all \(\j'\) present in (\ref{eq: Recurrence second expression}) satisfy \(\j' = \j\pm\h\), so we obtain the anticipated tridiagonal expression
	\[
	\Gamma_{\{m,m+1\}}^q\psi_{\mathbf{j}} = B_{\j}\psi_{\j-\h}+A_{\mathbf{j}}\psi_{\j}+C_{\j}\psi_{\j+\h}.
	\]
	Note the change in notation of the coefficients, this distinction is necessary when \(\j\) is no longer fixed. It remains to be proven that \(B_{\j}\) and \(C_{\j}\) are non-zero except in a number of special cases. To achieve this, we will act on \(\psi_{\j}\) with both sides of the equation (\ref{eq: q-BI relation with 2 generators 2}), use the obtained tridiagonal expression and gather the terms proportional to \(\psi_{\j}\). This results in the following recurrence relation:
	\begin{align}
	\label{eq: Three-term recurrence relation 1}
	\begin{split}
	& \left(2\lambda_m(\j)+(q+q^{-1})\lambda_m(\j+\h)\right) B_{\j+\h}C_{\j}
	+ \left(2\lambda_m(\j)+(q+q^{-1})\lambda_m(\j-\h)\right) B_{\j}C_{\j-\h} \\  = & \,\lambda_m(\j) + (q^{1/2}+q^{-1/2})\left([\mu_m]_q\lambda_{m-1}(\j) + [\mu_{m+1}]_q\lambda_{m+1}(\j)\right) \\ & + (q^{1/2}+q^{-1/2})^2\left([\mu_m]_q\lambda_{m+1}(\j)+[\mu_{m+1}]_q\lambda_{m-1}(\j)\right)A_{\j} -\left(2+q+q^{-1}\right)\lambda_m(\j)A_{\j}^2.
	\end{split}
	\end{align}
	This equation is to be compared with (3.26) in \cite{Genest&Vinet&Zhedanov-2016}. A second relation can be obtained using the Casimir element \(C_m\) from Lemma \ref{lemma: Casimir C_m}. We will act with \(C_m\) on the basis functions \(\psi_{\j}\). One easily finds from (\ref{eq: Alternative expression Casimir}) and Proposition \ref{prop: Spherical q-Dirac-Dunkl equation} that this yields
	\begin{align}
	\label{eq: Action of Casimir C}
	\begin{split}
	C_m\psi_{\j} = & \,\left(\lambda_{m-1}(\j)^2+\lambda_{m+1}(\j)^2 + \left([\mu_m]_q\right)^2 + \left([\mu_{m+1}]_q\right)^2 \right.\\ &\left.- (q-q^{-1})^2[\mu_m]_q[\mu_{m+1}]_q\lambda_{m-1}(\j)\lambda_{m+1}(\j) - \frac{q}{(1+q)^2}\right)\psi_{\j}.
	\end{split}
	\end{align}
	On the other hand, we already know that
	\begin{equation}
	\label{eq: Action of Gamma 1}
	\Gamma_{\{m,m+1\}}^q\psi_{\j} = B_{\j}\psi_{\j-\h}+ A_{\j}\psi_{\j} + C_{\j}\psi_{\j+\h}.
	\end{equation}
	Furthermore we have \[\Gamma_{[1;m-1]\cup\{m+1\}}^q = \{\Gamma_{[m]}^q,\Gamma_{\{m,m+1\}}^q\}_q - (q^{1/2}+q^{-1/2})\left(\Gamma_{\{m\}}^q\Gamma_{[m+1]}^q + \Gamma_{[m-1]}^q\Gamma_{\{m+1\}}^q\right), \]
	hence
	\begin{align}
	\label{eq: Action of Gamma 2}
	\begin{split}
	\Gamma_{[1;m-1]\cup\{m+1\}}^q\psi_{\j} = & \,(q^{1/2}+q^{-1/2})\left(A_{\j}\lambda_m(\j) - [\mu_m]_q\lambda_{m+1}(\j)-[\mu_{m+1}]_q\lambda_{m-1}(\j) \right)\psi_{\j} \\
	& \, + (q^{1/2}\lambda_m(\j-\h)+q^{-1/2}\lambda_m(\j))B_{\j}\psi_{\j-\h} \\
	& \, + (q^{1/2}\lambda_m(\j+\h)+q^{-1/2}\lambda_m(\j))C_{\j}\psi_{\j+\h}.
	\end{split}
	\end{align}
	Acting now with (\ref{def: Casimir operator C}) on \(\psi_{\j}\) and using (\ref{eq: Action of Gamma 1}) and (\ref{eq: Action of Gamma 2}), one finds:
	\begin{align}
	\label{eq: Action of Casimir C 2}
	\begin{split}
	C_m\psi_{\j} = \left[\left(q^{2\vert\j_{m-1}\vert+2\gamma_{[m]}}+q^{-1}\right)B_{\j+\h}C_{\j}+\left(q^{-2\vert\j_{m-1}\vert-2\gamma_{[m]}+2}+q^{-1}\right)B_{\j}C_{\j-\h}+ E_{\j}\right]\psi_{\j},
	\end{split}
	\end{align}
	where
	\begin{align*}
	E_{\j} = & \, qA_{\j}^2 + q^{-1}(q^{1/2}+q^{-1/2})^2\left(A_{\j}\lambda_m(\j)-[\mu_m]_q\lambda_{m+1}(\j)-[\mu_{m+1}]_q\lambda_{m-1}(\j)\right)^2 \\
	& + q\lambda_m(\j)^2 - (q^{-1/2}-q^{3/2})\left([\mu_m]_q[\mu_{m+1}]_q+\lambda_{m-1}(\j)\lambda_{m+1}(\j)\right)A_{\j} \\
	& - (q^{-1/2}-q^{3/2})\left([\mu_m]_q\lambda_{m-1}(\j)+[\mu_{m+1}]_q\lambda_{m+1}(\j)\right)\lambda_m(\j) \\
	& - (q^{1/2}-q^{-3/2})(q^{1/2}+q^{-1/2})\left(A_{\j}\lambda_m(\j) - [\mu_m]_q\lambda_{m+1}(\j)-[\mu_{m+1}]_q\lambda_{m-1}(\j) \right)\\
	& \phantom{-}\  \left([\mu_{m+1}]_q\lambda_{m-1}(\j)+[\mu_m]_q\lambda_{m+1}(\j)+q\lambda_m(\j)A_{\j}\right).
	\end{align*}
	Upon comparison with (\ref{eq: Action of Casimir C}), one finds a second expression that relates \(B_{\j}C_{\j-\h}\) and \(B_{\j+\h}C_{\j}\):
	\begin{align}
	\label{eq: Three-term recurrence relation 2}
	\begin{split}
	& \left(q^{2\vert\j_{m-1}\vert+2\gamma_{[m]}}+q^{-1}\right)B_{\j+\h}C_{\j}+\left(q^{-2\vert\j_{m-1}\vert-2\gamma_{[m]}+2}+q^{-1}\right)B_{\j}C_{\j-\h} \\ = & \, \lambda_{m-1}(\j)^2+\lambda_{m+1}(\j)^2 + \left([\mu_m]_q\right)^2 + \left([\mu_{m+1}]_q\right)^2 - \frac{q}{(1+q)^2} \\ & - (q-q^{-1})^2[\mu_m]_q[\mu_{m+1}]_q\lambda_{m-1}(\j)\lambda_{m+1}(\j)-E_{\j}
	\end{split}
	\end{align}
	Now observe from (\ref{def: eigenvalue}) that
	\begin{align*}
	2\lambda_m(\j)+(q+q^{-1})\lambda_m(\j+\h) & = (-1)^{\vert\j_{m-1}\vert+1}q^{-\vert\j_{m-1}\vert-\gamma_{[m]}+\frac12}\left(q^{2\vert\j_{m-1}\vert+2\gamma_{[m]}}+q^{-1}\right), \\
	2\lambda_m(\j) + (q+q^{-1})\lambda_m(\j-\h) & = (-1)^{\vert\j_{m-1}\vert+1}q^{-\vert\j_{m-1}\vert-\gamma_{[m]}+\frac12}\left(-q^{2\vert\j_{m-1}\vert+2\gamma_{[m]}-2}-q\right).
	\end{align*}
	Hence the equation (\ref{eq: Three-term recurrence relation 1}) can be rewritten as
	\begin{align}
	\label{eq: Three-term recurrence relation 3}
	\begin{split}
	& -\left(q^{2\vert\j_{m-1}\vert+2\gamma_{[m]}}+q^{-1}\right)B_{\j+\h}C_{\j} + \left(q^{2\vert\j_{m-1}\vert+2\gamma_{[m]}-2}+q\right)B_{\j}C_{\j-\h} \\
	= & \, (-1)^{\vert\j_{m-1}\vert}q^{\vert\j_{m-1}\vert + \gamma_{[m]}-\frac12} \left[
	\lambda_m(\j) + (q^{1/2}+q^{-1/2})\left([\mu_m]_q\lambda_{m-1}(\j) + [\mu_{m+1}]_q\lambda_{m+1}(\j)\right)\right. \\ & \left.+ (q^{1/2}+q^{-1/2})^2\left([\mu_m]_q\lambda_{m+1}(\j)+[\mu_{m+1}]_q\lambda_{m-1}(\j)\right)A_{\j} -\left(2+q+q^{-1}\right)\lambda_m(\j)A_{\j}^2\right].
	\end{split}
	\end{align}
	Adding the equations (\ref{eq: Three-term recurrence relation 2}) and (\ref{eq: Three-term recurrence relation 3}), the term in \(B_{\j+\h}C_{\j}\) will disappear and we are left with an explicit expression for \(B_{\j}C_{\j-\h}\), which after a lengthy calculation can be factored as
	\[
	B_{\j}C_{\j-\h} = \frac{S_{\j-\h}T_{\j}}{(q-q^{-1})^2},
	\]
	where we have defined \(S_{\j}\) and \(T_{\j}\) as
	\begin{align*}
	S_{\j} & = -\frac{\left(1+(-q)^{j_{m-1}}a_{\j}b_{\j}\right)\left(1-(-q)^{j_{m-1}}a_{\j}c_{\j}\right)\left(1-(-q)^{j_{m-1}}a_{\j}d_{\j}\right)\left(1-(-q)^{j_{m-1}-1}a_{\j}b_{\j}c_{\j}d_{\j}\right)}{a_{\j}\left(1+q^{2j_{m-1}-1}a_{\j}b_{\j}c_{\j}d_{\j}\right)\left(1-q^{2j_{m-1}}a_{\j}b_{\j}c_{\j}d_{\j}\right)}, \\
	T_{\j} & = \frac{a_{\j}\left(1-(-q)^{j_{m-1}}\right)\left(1-(-q)^{j_{m-1}-1}b_{\j}c_{\j}\right)\left(1-(-q)^{j_{m-1}-1}b_{\j}d_{\j}\right)\left(1+(-q)^{j_{m-1}-1}c_{\j}d_{\j}\right)}{\left(1-q^{2j_{m-1}-2}a_{\j}b_{\j}c_{\j}d_{\j}\right)\left(1+q^{2j_{m-1}-1}a_{\j}b_{\j}c_{\j}d_{\j}\right)},
	\end{align*}
	with
	\begin{align*}
	a_{\j} = q^{\gamma_m+\gamma_{m+1}-\frac12}, \quad b_{\j} = (-1)^{j_{m-1}+j_m+1}q^{-(j_{m-1}+j_m+\gamma_m+\gamma_{m+1}-\frac12)}, \\
	c_{\j} = (-1)^{j_{m-1}+j_m+1}q^{\vert\j_{m-2}\vert+\vert\j_{m}\vert+\gamma_{[m-1]}+\gamma_{[m+1]}-\frac12},  \quad d_{\j} = q^{\gamma_m-\gamma_{m+1}+\frac12}.
	\end{align*}
	In this factored form one readily checks that \(B_{\j}C_{\j-\h}\) vanishes only if \(\j\) or \(\j-\h\) is not \(k\)-allowable.
\end{proof}

In order to prove the irreducibility of the considered action, we would like to be able to map the basis function \(\psi_{\j}\) to \(\psi_{\j-\h}\) and \(\psi_{\j+\h}\). This can be done by means of the following projection operators. For \(\j-\h\) \(k\)-allowable, we define
\begin{equation}
\label{def: Projection operator j-h}
\mathbb{P}_{\j,\,\j-\h} = \left(\Gamma_{[m]}^q-\lambda_m(\j)\right)\left(\Gamma_{[m]}^q-\lambda_m(\j+\h)\right),
\end{equation}
and likewise, for \(\j+\h\) \(k\)-allowable,
\begin{equation}
\label{def: Projection operator j+h}
\mathbb{P}_{\j,\,\j+\h} = \left(\Gamma_{[m]}^q-\lambda_m(\j)\right)\left(\Gamma_{[m]}^q-\lambda_m(\j-\h)\right).
\end{equation}
\begin{lemma}
	\label{lemma: Irreducibility of the action}
	The combined action of the operator \(\Gamma_{\{m,m+1\}}^q\) and the projectors \(\mathbb{P}_{\j,\,\j-\h}\) and \(\mathbb{P}_{\j,\,\j+\h}\) on the functions \(\psi_{\j}\) is as follows:
	\begin{align}
	\begin{split}
	\mathbb{P}_{\j,\,\j-\h}\Gamma_{\{m,m+1\}}^q\psi_{\j} = \beta_m(\j-\h)\psi_{\j-\h}, \\
	\mathbb{P}_{\j,\,\j+\h}\Gamma_{\{m,m+1\}}^q\psi_{\j} = \gamma_m(\j+\h)\psi_{\j+\h},
	\end{split}
	\end{align}
	where the coefficient
	\begin{align*}
	\beta_m(\j-\h) & = (\lambda_m(\j-\h)-\lambda_m(\j))(\lambda_m(\j-\h)-\lambda_m(\j+\h))B_{\j}
	\end{align*}
	is non-zero if \(\j\) and \(\j-\h\) are both \(k\)-allowable, and
	\begin{align*}
	\gamma_m(\j+\h) & = (\lambda_m(\j+\h)-\lambda_m(\j))(\lambda_m(\j+\h)-\lambda_m(\j-\h))C_{\j}
	\end{align*}
	is non-zero if \(\j\) and \(\j+\h\) are both \(k\)-allowable.
\end{lemma}
\begin{proof}
	This follows immediately upon combining Proposition \ref{prop: Spherical q-Dirac-Dunkl equation} and Theorem \ref{prop: Three-term recurrence relation}.
\end{proof}

This allows us to prove the irreducibility of the considered action.

\begin{theorem}
	\label{thm: Action irreducible}
	The symmetry algebra \(\mathcal{A}_n^q\) acts irreducibly on the space \(\mathcal{M}_k^q(\mathbb{R}^n)\) of \(q\)-Dunkl monogenics.
\end{theorem}
\begin{proof}
A necessary and sufficient condition for the irreducibility is the ability to map any basis function \(\psi_{\j}\) to any other basis function \(\psi_{\j'}\). This can be done using the combined action of the projectors \(\mathbb{P}_{\j,\,\j\pm\h}\) and the operators \(\Gamma_{\{m,m+1\}}^q\) as in Lemma \ref{lemma: Irreducibility of the action}. First consider the case \(j_1'>j_1\). Applying consecutively \(\mathbb{P}_{\j,\,\j+\harg{2}}\Gamma_{\{2,3\}}^q\), \(\mathbb{P}_{\j+\harg{2},\,\j+2\harg{2}}\Gamma_{\{2,3\}}^q\),\dots,\(\mathbb{P}_{\j+(j_1'-j_1-1)\harg{2},\,\j+(j_1'-j_1)\harg{2}}\Gamma_{\{2,3\}}^q\) to the function \(\psi_{\j}\), we obtain
\[
\psi_{(j_1',j_1+j_2-j_1',j_3,\dots,j_{n-1})},
\]
up to a proportionality constant which is non-zero by Lemma \ref{lemma: Irreducibility of the action}. The case \(j_1'<j_1\) follows from a similar approach using the projector \(\mathbb{P}_{\j,\,\j-\harg{2}}\). If \(j_1' = j_1\) we may immediately move on to the next step. Similarly, using \(\mathbb{P}_{\j,\,\j\pm\harg{3}}\Gamma_{\{3,4\}}^q\), we map
\[
\psi_{(j_1',j_1+j_2-j_1',j_3,\dots,j_{n-1})} \mapsto \psi_{(j_1',j_2',j_1+j_2+j_3-j_1'-j_2',\dots,j_{n-1})},
\]
and so on. Note that the index \(j_{n-1}'\) will be fixed once \(j_{n-2}'\) is, by the constraint that \(\j'\) is \(k\)-allowable. This concludes the proof.
\end{proof}

\section*{Conclusions}
In this paper, we have constructed the higher rank \(q\)-deformed Bannai-Ito algebra \(\mathcal{A}_n^q\) in the multifold tensor product of \(\mathfrak{osp}_q(1\vert 2)\). By the established isomorphism in the rank one case, this also yields a higher rank version of the universal Askey-Wilson algebra. The main difficulty, namely the construction of the operators \(\Gamma_A^q\) for sets \(A\) with holes, has been overcome by the identification of a suitable coideal subalgebra and the corresponding coaction \(\tau\). We have derived a Casimir element for certain subalgebras of the higher rank \(q\)-Bannai-Ito and Askey-Wilson algebras, and we have shown that their generators satisfy analogs of the tridiagonal relations. We have also considered a realization of \(\mathfrak{osp}_q(1\vert 2)\) in terms of \(q\)-shift operators and reflections, which we have called the \(\mathbb{Z}_2^n\) \(q\)-Dirac-Dunkl model. We have realized \(\mathcal{A}_n^q\) inside this model and shown how it acts irreducibly on modules of null-solutions of the \(\mathbb{Z}_2^n\) \(q\)-Dirac-Dunkl operator.

In a next step, we would like to uncover further how \(\mathcal{A}_n^q\) is related to the orthogonal polynomials of the \(q\)-Askey scheme. More precisely, it would be of interest to consider a specific representation of \(\mathcal{A}_n^q\), derive several different bases for the representation space and compute their connection coefficients. These will be expressible in terms of Gasper \& Rahman's multivariate \((-q)\)-Racah polynomials, as defined in \cite{Gasper2}. The results of \cite{I1} will then lead us to a discrete realization for the higher rank \(q\)-Bannai-Ito algebra. This will be covered in our follow-up work \cite{DeBie&DeClercq-2019}. 

\section*{Acknowledgements}
This work was supported by the Research Foundation Flanders (FWO) under Grant EOS 30889451 and G.0116.13N. We wish to thank Toon Baeyens for writing Java code that was used to verify some of the formulae in the paper. We would also like to thank the referee for valuable comments and suggestions.

\section*{Appendix A: Expressions in the fourfold tensor product}
In this appendix we will give the explicit expression for the element \(\Gamma_A^q \in\ospq^{\otimes 4}\) for some of the non-trivial sets \(A\subset \{1,2,3,4\}\). These can be obtained using the definition of the coproduct \(\Delta\) and the coaction \(\tau\) as in (\ref{def: Coproduct}), (\ref{eq:Delta(Gamma)}) and (\ref{def:tau isomorphism}).

\begingroup
\allowdisplaybreaks
\begin{align*}
\Gamma_{\{1,3\}}^q = &  \,	
\frac{-q^{1/2}}{q - q^{-1}} K^{2} P \otimes 1 \otimes K^{2} P \otimes 1 - (q-q^{-1}) A_- A_+ P \otimes A_+ K \otimes A_- K \otimes 1 \\ & + A_- A_+ P \otimes 1 \otimes K^{2} P \otimes 1 -q^{1/2} A_- K^{-1} P \otimes K^{-2} P \otimes A_+ K \otimes 1 \\ & - q^{-1/2} A_- K^{-1} P \otimes A_+ K^{-1} P \otimes K^{-2} P \otimes 1\\& + q^{-1/2} A_+ K^{-1} P \otimes K^{2} P \otimes A_- K \otimes 1 + q^{-1/2} A_- K^{-1} P \otimes A_+ K^{-1} P \otimes K^{2} P \otimes 1 \\ & - (q-q^{-1}) A_- K^{-1} P \otimes A_+ K^{-1} P \otimes A_- A_+ P \otimes 1  \\ &+ \frac{q^{-1/2}}{q-q^{-1}} K^{-2} P \otimes 1 \otimes K^{-2} P \otimes 1 + K^{-2} P \otimes 1 \otimes A_- A_+ P \otimes 1 \\ & - (q-q^{-1}) A_- K^{-1} P \otimes A_+^{2} P \otimes A_- K \otimes 1 \\ & + q^{1/2} K^{2} P \otimes A_+ K \otimes A_- K \otimes 1- q^{1/2} K^{-2}P \otimes A_+ K \otimes A_- K \otimes 1.
\end{align*}
\endgroup

\begingroup
\allowdisplaybreaks
\begin{align*}
\Gamma_{\{1,4\}}^q = &\, \frac{q^{-1/2}}{q-q^{-1}} K^{-2} P \otimes 1 \otimes 1 \otimes K^{-2} P - \frac{{q}^{1/2}}{q-q^{-1}} K^{2} P \otimes 1 \otimes 1 \otimes K^{2} P\\ & - (q-q^{-1}) A_- A_+ P \otimes A_+ K \otimes K^{2} P \otimes A_- K + q^{1/2} K^{2} P \otimes A_+ K \otimes K^{2} P \otimes A_- K \\ & -q^{1/2} A_- K^{-1} P \otimes K^{-2} P \otimes K^{-2} P \otimes A_+ K \\& - (q-q^{-1}) A_- K^{-1} P \otimes K^{-2} P \otimes A_+ K^{-1}P \otimes A_- A_+ P \\ & -q^{1/2} K^{-2} P \otimes A_+ K \otimes K^{2} P \otimes A_- K + A_- A_+ P \otimes 1 \otimes 1 \otimes K^{2} P \\ &
-(q-q^{-1}) A_- K^{-1} P \otimes A_+ K^{-1} P \otimes 1 \otimes A_- A_+ P\\&  - (q-q^{-1}) A_- K^{-1} P \otimes A_+^{2} P \otimes K^{2} P \otimes A_- K\\& - (q-q^{-1}) A_- K^{-1} P \otimes K^{-2} P \otimes A_+^{2} P \otimes A_- K - (q-q^{-1}) A_- A_+ P \otimes 1 \otimes A_+ K \otimes A_- K \\ & - q^{-1/2} A_- K^{-1} P \otimes A_+ K^{-1} P \otimes 1 \otimes K^{-2} P+ q^{-1/2} A_- K^{-1} P \otimes A_+ K^{-1} P \otimes 1 \otimes K^{2} P\\& - q^{-1/2} A_- K^{-1} P \otimes K^{-2} P \otimes A_+ K^{-1} P \otimes K^{-2} P\\& + q^{-1/2}A_+ K^{-1} P \otimes K^{2} P \otimes K^{2} P \otimes A_- K + K^{-2} P \otimes 1 \otimes 1 \otimes A_- A_+ P \\ & + (q-q^{-1})(q^{1/2}-q^{-1/2}) A_- K^{-1} P \otimes A_+ K^{-1} P \otimes A_+ K \otimes A_- K \\ & + q^{-1/2} A_- K^{-1} P \otimes K^{-2} P \otimes A_+ K^{-1} P \otimes K^{2} P  \\ & + q^{1/2} K^{2} P \otimes 1 \otimes A_+ K \otimes A_- K - q^{1/2} K^{-2} P \otimes 1 \otimes A_+ K \otimes A_- K.
\end{align*}
\endgroup

\begingroup
\allowdisplaybreaks
\begin{align*}
\Gamma_{\{1,2,4\}}^q = &\, -q^{-1/2} A_- K^{-1} P \otimes 1 \otimes A_+ K^{-1} P \otimes K^{-2} P \\ & + q^{1/2}(q-q^{-1}) A_- K^{-1} P \otimes A_+ K \otimes A_+ K \otimes A_- K \\ & -q^{1/2} A_- K^{-1} P \otimes A_+ K \otimes 1 \otimes K^{2} P-q^{1/2} K^{-2} P \otimes A_- K^{-1} P \otimes K^{-2} P \otimes A_+ K \\ & + q^{1/2} K^{-2} P \otimes A_- K^{-1} P \otimes A_+ K^{-1} P \otimes K^{2} P\\& -q^{1/2} A_- K^{-1} P \otimes 1 \otimes K^{-2} P \otimes A_+ K \\ & - (q-q^{-1}) K^{-2} P \otimes A_- A_+ P \otimes A_+ K \otimes A_- K \\& - (q-q^{-1})A_- K^{-1} P \otimes 1 \otimes A_+ K^{-1} P \otimes A_- A_+ P \\ & + K^{-2} P \otimes K^{-2} P \otimes 1 \otimes A_- A_+ P - \frac{q^{1/2}}{q-q^{-1}} K^{2} P \otimes K^{2} P \otimes 1 \otimes K^{2} P \\ & + q^{-1/2} A_- K^{-1} P \otimes 1 \otimes A_+ K^{-1} P \otimes K^{2} P \\ & - (q-q^{-1}) A_- K^{-1} P \otimes 1 \otimes A_+^{2} P \otimes A_- K \\ & - (q-q^{-1}) K^{-2} P \otimes A_- K^{-1} P \otimes A_+ K^{-1} P \otimes A_- A_+ P \\ & - q^{-1/2} K^{-2} P \otimes A_- K^{-1} P \otimes A_+ K^{-1} P \otimes K^{-2} P\\ & + q^{-1/2} K^{-2} P \otimes A_+ K^{-1} P \otimes K^{2} P \otimes A_- K \\ & + \frac{q^{-1/2}}{q-q^{-1}} K^{-2} P \otimes K^{-2} P \otimes 1 \otimes K^{-2} P \\ &-q^{1/2} K^{-2} P \otimes K^{-2} P \otimes A_+ K \otimes A_- K \\ & + q^{1/2} K^{2} P \otimes K^{2} P \otimes A_+ K \otimes A_- K + A_- A_+ P \otimes K^{2} P \otimes 1 \otimes K^{2} P \\ & - q^{-1/2}(q-q^{-1}) A_+ K^{-1} P \otimes A_- K \otimes A_+ K \otimes A_- K \\ & - (q-q^{-1}) K^{-2} P \otimes A_- K^{-1} P \otimes A_+^{2} P \otimes A_- K \\& - (q-q^{-1}) A_- A_+ P \otimes K^{2} P \otimes A_+ K \otimes A_- K  \\& +  K^{-2} P \otimes A_- A_+ P \otimes 1 \otimes K^{2} P + q^{-1/2} A_+ K^{-1} P \otimes 1 \otimes K^{2} P \otimes A_- K \\ & + q^{-1/2}A_+ K^{-1} P \otimes A_- K \otimes 1 \otimes K^{2} P.
\end{align*}
\endgroup


\end{document}